\documentclass[11p, english]{amsart}
\usepackage{amsfonts, amsmath, amssymb, amscd, amsthm, array, graphicx}
\usepackage[latin9]{inputenc}
\usepackage{tikz}
\usepackage{amsmath}
\usepackage{faktor}
\usepackage{graphicx}
\usepackage{amssymb}
\usepackage{babel}
\usepackage{color}
\usepackage[all]{xy}

\usepackage{mathtools}
\usepackage{hyperref,url}

\makeatletter
\def\blfootnote{\xdef\@thefnmark{}\@footnotetext}
\makeatother
\usepackage{tikz}
\usetikzlibrary{arrows,shapes,snakes,automata,backgrounds,petri}
\usetikzlibrary{automata} 
\usetikzlibrary[automata] 
\usepackage{geometry}
\usepackage{pdflscape}
\usepackage{graphicx}
\usetikzlibrary{external,automata,trees,positioning,shadows,arrows,shapes.geometric}

\newlength\Textht
\newtheorem{thm}{Theorem}[section]
\newtheorem{cor}[thm]{Corollary}
\newtheorem{lemma}[thm]{Lemma}
\newtheorem{prop}[thm]{Proposition}
\newtheorem{obs}[thm]{Observation}
\newtheorem{que}[thm]{Question}
\newtheorem{conj}[thm]{Conjecture}
\newtheorem{claim}[thm]{Claim}

\theoremstyle{definition}
\newtheorem{dfn}[thm]{Definition}

\newtheorem{rem}[thm]{Remark}
\newtheorem{ex}[thm]{Example}

\newcommand{\dc}{\operatorname{dc}}
\newcommand{\rk}{\operatorname{r}}
\newcommand{\rrk}{\operatorname{\tilde{r}}}
\newcommand{\trk}{\operatorname{\tilde{r}}}
\newcommand{\di}{\operatorname{di}}
\newcommand{\alge}{\mathcal{AE}}
\newcommand{\Z}{{\mathbb Z}}
\newcommand\ZZ{\mathbb{Z}}
\newcommand{\N}{{\mathbb N}}

\newcommand{\GL}{\operatorname{GL}}
\newcommand{\diag}{\operatorname{diag}}
\newcommand{\aut}{\operatorname{Aut}}

\newcommand{\fix}{\operatorname{Fix}}
\newcommand{\row}{\operatorname{row}}

\newcommand{\bp}{\odot}

\newcommand{\core}{\operatorname{core}}

\begin{document}

\title{Degrees of compression and inertia for free-abelian times free groups}

\author{Mallika Roy}
\address{Departament de Matem\`atiques, Universitat Polit\`ecnica de Catalunya, Catalonia}
\email{mallika.roy@upc.edu}

\author{Enric Ventura}
\address{Departament de Matem\`atiques, Universitat Polit\`ecnica de Catalunya, Catalonia} \email{enric.ventura@upc.edu}



\date{\today}

\begin{abstract}
We introduce the concepts of degree of inertia, $\di_G(H)$, and degree of compression, $\dc_G(H)$, of a finitely generated subgroup $H$ of a given group $G$. For the case of direct products of free-abelian and free groups, we compute the degree of compression and give an upper bound for the degree of inertia. Imposing some technical assumptions to the supremum involved in the definition of degree of inertia, we introduce the notion called restricted degree of inertia, $\di'_G(H)$, and, again for the case $\Z^m \times F_n$, we provide an explicit formula relating it to the restricted degree of inertia of its projection to the free part, $\di'_{F_n}(H\pi)$.
\end{abstract}

\maketitle


\section{Introduction}

For a group $G$, we write $\rk(G)$ to denote the rank of $G$, i.e., the minimum cardinal of a generating set for $G$. To work with group morphisms, we use the notational convention of writing arguments on the left, i.e., $\phi\colon G_1\to G_2$, $g\mapsto g\phi$; and so, compositions as written: $g\phi\psi =(g\phi)\psi$. Accordingly, we write conjugations on the right, $H^g =g^{-1}Hg$, and commutators in the form $[a,b]= a^{-1}b^{-1}ab$. For a subgroup $H\leqslant G$, we shall use the notation $H\leqslant_{fg} G$ to emphasize that $H$ is finitely generated, and $H\leqslant_{fi} G$ (resp., $H\leqslant_{\infty} G$, or $H\leqslant_l G$) to emphasize it is of finite index (resp., infinite index, or index $l$) in $G$.

In the commutative realm, the rank function is increasing in the sense that $H\leqslant K\leqslant G$ implies $\rk(H)\leqslant \rk(K)$. This is far from true in general, and the main expression of this phenomena can be found in the context of free groups $F_n$, where the free group of countably infinite rank easily embeds into the free group of rank 2, $F_{\aleph_0}\leqslant F_2$. However, when restricting ourselves to certain families of groups and subgroups, the rank function tends to behave less wildly and somehow closer to the commutative behaviour. An example of this situation is again in finitely generated free groups, but restricting our attention to subgroups fixed by automorphisms or endomorphisms: the story began in~\cite{DS}, where Dyer--Scott showed that $\fix(\varphi)$ is a free factor of $F_n$ for every finite order automorphism $\varphi \in \aut(F_n)$, and conjectured that $\rk(\fix(\varphi))\leqslant n$, in general. This was proved later by Bestvina--Handel~\cite{BH}, and extended several times in subsequent papers, all of them pointing to the direction that the rank function, when restricted to subgroups fixed by endomorphisms, tends to behave similarly to the abelian case. In this spirit, the following concepts were first introduced by Dicks--Ventura in~\cite{DV} and turned out to be quite relevant in the subsequent literature:

\begin{dfn}
Let $G$ be a group. A finitely generated subgroup $H{\leqslant}_{fg}G$ is said to be \emph{compressed in $G$} if $\rk(H)\leqslant \rk(K)$, for every $H\leqslant K\leqslant G$. And $H$ is said to be \emph{inert in $G$} if $\rk(H\cap K)\leqslant \rk(K)$, for every $K\leqslant G$. (Note that, equivalently, in both definitions one can restrict the attention to those subgroups $K$'s being finitely generated.)
\end{dfn}

Observe that, directly by induction from the definition, inert subgroups are closed under finite intersections. Also, inert subgroups are compressed, while the other implication is not known to be true or false in free groups (this is the so-called inertia conjecture, see Zhang--Ventura--Wu~\cite{ZVW} for partial results), and it is not true in general:

\begin{ex}
Consider the direct product of the Klein bottle group with the group of integers, say $G=\langle a,b \mid bab^{-1}a\rangle \times \langle c \mid \, \rangle$, and its subgroup $H=\langle a, b^2, c\rangle \simeq \ZZ^3$. By Corollary~4.3 and Proposition 4.4 from~\cite{ZVW}, $H$ is compressed but not inert in $G$.
\end{ex}

Several known results involving these concepts include the following:

\begin{thm}
\begin{itemize}
\item[(i)] (Dicks--Ventura, \cite{DV}): Arbitrary intersections of fixed subgroups of injective endomorphisms of $F_n$ are inert in $F_n$;
\item[(ii)] (Martino--Ventura, \cite{MV}): arbitrary intersections of fixed subgroups of endomorphisms of $F_n$ are compressed in $F_n$;
\item[(iii)] (Wu--Zhang, \cite{WZ}): arbitrary intersections of fixed subgroups of automorphisms of closed surface groups $G$ with negative Euler characteristic are inert in $G$;
\item[(iv)] (Wu--Ventura--Zhang, \cite{WVZ}): arbitrary intersections of fixed subgroups of endomorphisms of surface groups $G$ are compressed in $G$.
\end{itemize}
\end{thm}

Also, in~\cite{WVZ} and~\cite{ZVW}, Zhang--Ventura--Wu studied similar questions within the family of finite direct products of free and surface groups, where more interesting phenomena show up.

In the present paper we introduce a quantification for these two concepts and study it within the families of free groups, and free-abelian times free groups. For technical reasons it is better to work with the so-called \emph{reduced rank} of a group $G$, defined as $\rrk(G)=\max\{ 0, \rk(G)-1\}$, i.e., one unit less than the rank except for the trivial group for which we take zero (note that then, $\rrk(1)=\rrk(\Z)=0$ while $0=\rk(1)\neq \rk(\Z)=1$). Observe that $H\leqslant G$ is compressed in $G$ if and only if $\rrk(H)/\rrk(K)\leqslant 1$ for every $H\leqslant K\leqslant_{fg} G$; and that $H\leqslant G$ is inert in $G$ if and only if $\rrk(H\cap K)/\rrk(K)\leqslant 1$ for every $K\leqslant_{fg} G$ (understanding always $0/0=1$). This motivates the following quantitative definitions:

\begin{dfn}
Let $G$ be a group and $H\leqslant_{fg}G$. The \emph{degree of compression of $H$ in $G$} is $\dc_G(H)=\sup_{K} \{ \rrk(H)/\rrk(K)\}$, where the supremum is taken over all subgroups $H\leqslant K{\leqslant}_{fg} G$. Similarly, the \emph{degree of inertia of $H$ in $G$} is $\di_G(H)=\sup_{K} \{ \rrk(H\cap K)/\rrk(K)\}$, where the supremum is taken over all $K{\leqslant}_{fg} G$ satisfying $H\cap K\leqslant_{fg} G$ and, in both cases, $0/0$ is understood to be 1.
\end{dfn}

Note that, taking $K=H$, we get $\dc_G(H)\geqslant 1$ and $\di_G(H)\geqslant 1$. So, the possibility of $K$ being cyclic (which leads in both cases to $0/0=1$) is irrelevant in both definitions and we can restrict the two supremums to non-cyclic $K$'s without changing their final values.

Note also that the supremum in the definition of degree of compression is always a maximum, since the numerator has a fixed value and the denominator takes only natural values. Although we do not have any particular example, the supremum in the definition of degree of inertia could, in principle, not be attained at any particular subgroup $K$. In this sense, the following is an intriguing question for which, at the time of writing, we have no idea how to answer:

\begin{que}
Is there a (finitely generated) group $G$ and a subgroup $H\leqslant_{fg} G$ such that $\di_G(H)$ is a transcendental number? or an irrational number? Or such that the supremum in $\di_G(H)$ is not a maximum?
\end{que}

Observe that in the definition of degree of inertia, we take the supremum only over those subgroups $K\leqslant_{fg} G$ \emph{whose intersection with $H$ is again finitely generated}. In groups $G$ with the Howson property (the intersection of any two finitely generated subgroups is again finitely generated), like free groups or surface groups, this is no restriction at all and that supremum is over all finitely generated $K$'s. Otherwise, if $G$ is not Howson we are eliminating, on purpose, those possible finitely generated $K$'s having non-finitely generated intersection with $H$ (which would force $\di_G(H)$ to be automatically infinite). However observe that, even with the present limited definition, $\di_G(H)$ may be infinite as well; explicit examples will be shown later.

We adapt the definition of inertia to the non-Howson environments by saying that a subgroup $H\leqslant G$ is \emph{finitary inert in $G$} if $\rk(H\cap K)\leqslant \rk(K)$ for every $K\leqslant_{fg} G$ such that $H\cap K\leqslant_{fg} G$. The following observation then follows directly from the definitions and presents the values of $\dc_G(H)$ and $\di_G(H)$ as a quantification of how far is the subgroup $H\leqslant_{fg} G$ from being compressed and from being finitary inert in $G$, respectively:

\begin{obs}\label{elementary}
Let $G$ be a group and $H\leqslant_{fg}G$. Then,
\begin{itemize}
\item[(i)] $1\leqslant \dc_G(H)\leqslant \di_G(H)$;
\item[(ii)] $\dc_G(H)=1$ if and only if $H$ is compressed in $G$;
\item[(iii)] $\di_G(H)=1$ if and only if $H$ is finitary inert in $G$.
\end{itemize}
\end{obs}

The following intriguing question is open, as far as we know:

\begin{que}
Is there a (finitely generated) group $G$ with a subgroup $H\leqslant_{fg} G$ being finitary inert but not inert? (i.e., satisfying $\rrk(H\cap K)\leqslant \rrk(K)$ for every $K\leqslant_{fg}G$ with $H\cap K\leqslant_{fg} G$, but simultaneously admitting some $K_0\leqslant_{fg} G$ with $\rrk(H\cap K_0)=\infty$?).
\end{que}

In the present paper we study these notions for the case of the free group and obtain the following result in Section~\ref{free}:

\begin{thm}\label{main1}
For any finitely generated free group $G=F_n$, the function $\dc_{F_n}(\cdot)$ is computable; more precisely, there is an algorithm which, on input $h_1,\,\ldots ,h_r\in F_n$, it computes the value of $\dc_{F_n} (\langle h_1, \ldots ,h_r\rangle)$ and outputs a free basis of a subgroup $K\leqslant_{fg} F_n$ where it is attained.
\end{thm}

The question whether $\di_{F_n}(\cdot)$ is computable in free groups (related to the question whether the corresponding supremum is a maximum or not) seems to be much more delicate. In Section~\ref{free} we refer to a quite similar question, which was successfully solved recently by S. Ivanov in~\cite{Ivanov}. However, at the time of writing, we do not know how to use this result to eventually compute $\di_{F_n}(\cdot)$.

Then, we concentrate in free-abelian times free groups, $G=\Z^m\times F_n$, where the situation is richer and trickier because, for $m\geqslant 1$, $n\geqslant 2$, $G$ is known to be non-Howson (the easy example $G=\Z \times F_2=\langle t\rangle\times \langle a,b\rangle$, $H=\langle a,b\rangle$, $K=\langle ta, b\rangle$, and $H\cap K=\langle w(a,b) \mid |w|_a=0\rangle =\langle a^{-n}ba^n, n\in \Z\rangle$ already appears in~\cite{BK} attributed to Moldavanskii). Denoting by $\pi\colon G\twoheadrightarrow F_n$ the natural projection, in Section~\ref{fatf-c} we study the degree of compression and prove the following result:

\begin{thm}\label{main-C}
For any given $H\leqslant_{fg} G=\Z^m\times F_n$, and using the notation from Section~\ref{fatf-c}, we have
 $$
\dc_G(H)=\rrk(H) \,\,\left/ \min_{J\in \mathcal{AE}_{F_n}(H\pi)} \{\rrk(J)+d(A,B,U_J)\}.\right.
 $$
Moreover, $\dc_G(H)$ is algorithmically computable; more precisely, there is an algorithm which, on input $h_1,\ldots ,h_r\in G$, it computes the value of $\dc_G(\langle h_1, \ldots ,h_r\rangle)$ and outputs a basis of a subgroup $K\leqslant_{fg} G$ where that supremum is attained.
\end{thm}

In Section~\ref{fatf-i}, we study the degree of inertia, also for groups of the form $G=\Z^m\times F_n$, and get the following result:

\begin{thm}\label{mainI}
Let $H{\leqslant}_{fg} \, G=\Z^m\times F_n$, and let $L_H=H\cap \Z^m$.
\begin{itemize}
\item[(i)] If $\rk(H\pi)\leqslant 1$ then $\di_G(H)=1$;
\item[(ii)] if $\rk(H\pi)\geqslant 2$ and $[\Z^m : L_H]=\infty$ then $\di_G(H)=\infty$;
\item[(iii)] if $\rk(H\pi)\geqslant 2$ and $[\Z^m : L_H ]=l<\infty$ then $\di_G(H)\leqslant l\di_{F_n}(H\pi)$.
\end{itemize}
\end{thm}

We conjecture that the inequality in Theorem~\ref{mainI}~(iii) is, in fact, an equality, i.e.,

\begin{conj}\label{=}
For $G=\Z^m\times F_n$ and $H{\leqslant}_{fg} \, G$ with $\rk(H\pi)\geqslant 2$, we have the equality $\di_G(H)=[\Z^m : L_H ]\cdot \di_{F_n}(H\pi)$, where $L_H=H\cap \Z^m$.
\end{conj}

Unfortunately, we cannot complete a general proof for this equality. Instead, we get an approximation to it by introducing a technical modification to the definition of $\di_G(\cdot)$, the so-called \emph{restricted degree of inertia}:

\begin{dfn}\label{res def di}
Let $G$ be a group and $\pi\colon G \twoheadrightarrow G/Z(G)$ the projection modulo its center $Z(G)\unlhd G$. Let $H\leqslant_{fg} G$ be such that $H\pi$ is not virtually cyclic and $H\pi \nleqslant [G\pi, G\pi]$. The \textit{restricted degree of inertia of $H$ in $G$} is $\di'_G(H) =\sup_K\{\trk(H \cap K)/\trk(K)\}$, where the supremum is taken over all $K\leqslant_{fg}G$ satisfying $H \cap K \leqslant_{fg} G$, $[H\pi : H\pi \cap K\pi] = \infty$, and $H\pi \cap K\pi \nleqslant [G\pi, G\pi]$ (again, understanding $0/0=1$).
\end{dfn}

\begin{rem}
The conditions on the projection of the subgroup $H\pi$ not to be virtually cyclic and not to be contained in the commutator $[G\pi, G\pi]$ are just to make sure the supremum is not over the empty set: assuming these two conditions, let $h\in H$ be such that $h\pi\notin [G\pi, G\pi]$ and take $K=\langle h \rangle\leqslant G$; clearly, $K=H\cap K=\langle h \rangle$ is cyclic and so finitely generated, $H\pi\cap K\pi =\langle h\pi \rangle \leqslant_\infty H\pi$ because $H\pi$ is not virtually cyclic, and $h\pi\in H\pi\nleqslant [G\pi, G\pi]$. Moreover, $\trk(H \cap K)/\trk(K)=0/0=1$ and so, $\di'_G(H)\geqslant 1$.

Note that, the two cases of our interest are the following: (i) $G=F_n$, $n\geqslant 2$, we have $Z(G)=1$, $G/Z(G)=F_n$, and the definition applies to any non-cyclic subgroup $H\leqslant G$ such that $H\not\leqslant [F_n, F_n]$; and (ii) $G=\Z^m \times F_n$, $n\geqslant 2$, we have $Z(G)=\Z^m$, $G/Z(G)=F_n$, and the definition applies to any subgroup $H\leqslant G$ such that $H\pi$ is not cyclic and $H\pi \not\leqslant [F_n, F_n]$.
\end{rem}

Observe that the definition of restricted degree of inertia coincides with that of degree of inertia, except for the extra technical conditions required to the subgroups $K$ over which the supremum runs. Hence, $1\leqslant \di'_G(H)\leqslant \di_G(H)$ and we conjecture that, at least in the free and free-abelian times free cases, they do coincide.

\begin{conj}\label{di'=di}
For $G=F_n$ and $G=\Z^m\times F_n$, $\di'_G(\cdot )=\di_G(\cdot )$.
\end{conj}

Our main result in Section~\ref{section for rdi} is the desired equality from Theorem~\ref{mainI}~(iii), but expressed with the restricted degrees of inertia:

\begin{thm}\label{mainII}
Let $H{\leqslant}_{fg} G=\Z^m\times F_n$ be such that $H\pi$ is not cyclic and $H\pi \nleqslant [F_n, F_n]$, and let $L_H=H\cap \Z^m$.
\begin{enumerate}
\item[(i)] If $[\Z^m :L_H] =\infty$ then $\di_G'(H)=\infty$;
\item[(ii)] if $[\Z^m :L_H] =l$ then $\di_G'(H)=l\di_{F_n}'(H\pi)$.
\end{enumerate}
\end{thm}

Theorem~\ref{mainII} is the most involved and technical result in the paper. We hope in the future some new ideas come up allowing to avoid the technical working conditions (namely, $H\pi\cap K\pi$ having infinite index in $H\pi$, and not being contained in $[F_n,F_n]$) and to recycle the proof of Theorem~\ref{mainII} into a proof for Conjecture~\ref{=} (or, better, for Conjecture~\ref{di'=di}).

We conclude the present section with a straightforward result which will be used later.

\begin{lemma}\label{isomorphism}
Let $\phi \colon G_1\to G_2$ be an isomorphism of groups. For every $H\leqslant_{fg} G_1$,
\begin{itemize}
\item[(i)] $\dc_{G_2}(H\phi)=\dc_{G_1}(H)$;
\item[(ii)] $\di_{G_2}(H\phi)=\di_{G_1}(H)$;
\item[(iii)] with the extra assumptions that $H\pi_1$ is not virtually cyclic and $H\pi_1 \nleqslant [G_1\pi_1, G_1\pi_1]$, where $\pi_i \colon G_i \to G_i/Z(G_i)$ is the natural projection modulo the center $Z(G_i)\unlhd G_i$, $i=1,2$, we have $\di'_{G_2}(H\phi)=\di'_{G_1}(H)$.
\end{itemize}
\end{lemma}

\begin{proof}
For every $K\leqslant_{fg} G_1$ with $H\leqslant K$, we have $K\phi\leqslant_{fg}G_2$ and $H\phi\leqslant K\phi$ so, $\rrk(H)=\rrk(H\phi)\leqslant \dc_{G_2}(H\phi)\cdot \rrk(K\phi)= \dc_{G_2}(H\phi)\cdot \rrk(K)$. Therefore, $\dc_{G_1}(H)\leqslant \dc_{G_2}(H\phi)$. By symmetry, we get (i).

Similarly, for every $K\leqslant_{fg} G_1$ with $H\cap K\leqslant_{fg} G_1$, we have $K\phi\leqslant_{fg}G_2$ and $H\phi \cap K\phi =(H\cap K)\phi\leqslant_{fg} G_2$ so, $\rrk(H\cap K)=\rrk((H\cap K)\phi)=\rrk(H\phi\cap K\phi)\leqslant \di_{G_2}(H\phi)\cdot \rrk(K\phi)=\di_{G_2}(H\phi)\cdot \rrk(K)$. Therefore, $\di_{G_1}(H)\leqslant \di_{G_2}(H\phi)$. By symmetry, we deduce (ii).

The argument in the previous paragraph also shows (iii), provided we see that the technical conditions in the supremum of the definition of restricted degree of inertia get preserved under $\phi$. And in fact they do: suppose $K\leqslant_{fg} G_1$ is such that $H \cap K \leqslant_{fg} G_1$, $H\pi_1 \cap K\pi_1 \leqslant_{\infty} H\pi_1$ and $H\pi_1 \cap K\pi_1 \nleqslant [G_1\pi_1, G_1\pi_1]$; since, $\phi$ is an isomorphism, $K\phi \leqslant_{fg} G_2$ and $H\phi \cap K\phi = (H \cap K)\phi \leqslant_{fg} G_2$; also, since $\phi$ maps $Z(G_1)$ onto $Z(G_2)$, there exists an isomorphism $\bar{\phi} \colon G_1/Z(G_1) \to G_2/Z(G_2)$ such that $\pi_1\bar{\phi}=\phi \pi_2$ and, hence, $H\pi_1 \cap K\pi_1 \leqslant_{\infty} H\pi_1$ implies $H\phi \pi_2 \cap K\phi \pi_2 = H\pi_1 \bar{\phi} \cap K\pi_1 \bar{\phi}= (H\pi_1 \cap K\pi_1)\bar{\phi} \leqslant_\infty H\pi_1 \bar{\phi} = H\phi\pi_2$; finally, $H\phi \pi_2 \cap K\phi \pi_2 = H\pi_1\bar{\phi} \cap K\pi_1\bar{\phi} = (H\pi_1 \cap K\pi_1)\bar{\phi} \nleqslant [G_1\pi_1, G_1\pi_1]\bar{\phi} = [G_1\phi\pi_2, G_1\phi\pi_2] = [G_2 \pi_2, G_2\pi_2]$. This completes the proof of the lemma.
\end{proof}

\begin{cor}\label{conjugation}
Let $G$ be a group. For every $H\leqslant_{fg} G$ and every $g\in G$, $\dc_G(H^g)=\dc_G(H)$, $\di_G(H^g)=\di_G(H)$, and (with the extra assumptions on $H$ and so, on $H^g$) $\di'_G(H^g)=\di'_G(H)$. \qed
\end{cor}

\section{The free case}\label{free}

For all the paper, we fix an alphabet of $n$ letters, $X=\{x_1, \ldots ,x_n\}$, and consider the free group on it, $F(X)$, also denoted by $F_n$. In the present section we study the degrees of compression and inertia in the context of the free group, i.e., the functions $\dc_{F_n}(\cdot)$ and $\di_{F_n}(\cdot)$.

Hanna Neumann proved in~\cite{HN} that $\rrk(H\cap K)\leqslant 2\rrk(H)\rrk(K)$, for every $H,K\leqslant_{fg} F_n$. And the same assertion removing the factor ``2" became soon known as the \emph{Hanna Neumann conjecture}. This has been a major problem in Geometric Group Theory, with lots of partial results and improvements appearing in the literature since then. An interesting one was done by W. Neumann in~\cite{WN}, who proved the stronger fact $\sum_{x\in S} \rrk(H\cap K^s)\leqslant 2\rrk(H)\rrk(K)$ (known as the strengthened Hanna Neumann inequality), where $S$ is any set of double coset representatives of $F_n$ modulo $H$ on the left and $K$ on the right (i.e., $S\subseteq F_n$ contains one and only one element in each double coset $H\backslash F_n/K$); in particular, this implies that, for all $H,K\leqslant_{fg} F_n$, \emph{all except finitely many} of the intersections $H\cap K^x$, $x\in F_n$, are trivial or cyclic. Few years ago the Hanna Neumann conjecture, even in its strengthened version, has been completely resolved in the positive, independently by J. Friedman~\cite{Friedman} and by I. Mineyev~\cite{Mineyev} (see also W. Dicks~\cite{Dicks}). This can be interpreted as the following upper bound for $\dc_{F_n}(H)$ and $\di_{F_n}(H)$ in terms of the subgroup $H\leqslant_{fg}F_n$:

\begin{obs}
For $H\leqslant_{fg} F_n$, we have $1\leqslant \dc_{F_n} (H)\leqslant \di_{F_n}(H)\leqslant \rrk (H)$.
\end{obs}

Friedman--Mineyev's inequality is easily seen to be tight (consider, for example, the subgroups $H=\langle a, b^{-1}ab\rangle$ and $K=\langle b, a^2, aba\rangle$ of $F_2$, and its intersection $H\cap K=\langle a^2, b^{-1}a^2b, b^{-1}aba\rangle$); therefore, it can be interpreted in the following way: ``the smallest possible multiplicative constant $\alpha\in \mathbb{R}$ satisfying $\rrk(H\cap K)\leqslant \alpha \rrk(H)\rrk(K)$, for every $H,K\leqslant_{fg} F_n$, is $\alpha=1$". Now fix the subgroup $H$: by definition, the smallest possible constant $\alpha\in \mathbb{R}$ satisfying $\rrk(H\cap K)\leqslant \alpha \rrk(H)\rrk(K)$, for every $K\leqslant_{fg} F_n$, is $\alpha=\di_{F_n}(H)/\rrk(H)$.

S. Ivanov~\cite{Ivanov} already considered and studied the strengthened version of what we call here the degree of inertia. He defined the \emph{Walter Neumann coefficient} of $H\leqslant_{fg} F_n$ as $\sigma(H):= \sup_{K\leqslant_{fg}F_n} \rrk(H,K)/\rrk(H)\rrk(K)$, where $\rrk(H,K)=\sum_{s\in H\backslash F_n/K} \rrk(H\cap K^s)$ (understanding $0/0=1$). In other words, $\sigma(H)$ is the smallest possible constant $\alpha\in \mathbb{R}$ such that $\rrk(H, K)\leqslant \alpha \rrk(H)\rrk(K)$, for every $K\leqslant_{fg} F_n$. Using linear programming techniques, Ivanov was able to prove the following remarkable result:

\begin{thm}[Ivanov, \cite{Ivanov}]
For any finitely generated free group $F_n$, the function $\sigma$ is computable and the supremum is a maximum; more precisely, there is an algorithm which, on input $h_1,\ldots ,h_r\in F_n$, it computes the value of $\sigma(\langle h_1, \ldots ,h_r\rangle)$ and outputs a free basis of a subgroup $K\leqslant_{fg} F_n$ where that supremum is attained.
\end{thm}

Ivanov's proof is involved and technical. Although it looks quite similar, we have been unable to adapt Ivanov's arguments to answer any of the following questions which, as far as we know, remain open:

\begin{que}
Is the function $\di_{F_n}(\cdot)$ computable? Is that supremum always a maximum? More precisely, is there an algorithm which, on input $h_1,\ldots ,h_r\in F_n$, computes the value of $\di_{F_n} (\langle h_1, \ldots ,h_r\rangle)$? Or even more, it outputs a free basis of a subgroup $K\leqslant_{fg} F_n$ where it is attained?
\end{que}

The corresponding questions for the degree of compression are much easier and can be established with the use of Stallings' graphs (we assume the reader is familiar with these techniques), algebraic extensions, and Takahasi's Theorem.

\begin{dfn}
Let $H\leqslant_{fg} K\leqslant_{fg} F_n$. If $H$ is a free factor of $K$ we write $H\leqslant_{ff}K$. On the other extreme, the extension $H\leqslant K$ is said to be \emph{algebraic}, denoted by $H\leqslant_{alg} K$, if $H$ is not contained in any proper free factor of $K$, i.e., if $H\leqslant A\leqslant K=A*B$ implies $B=1$; we denote by $\alge_{F_n}(H)$ the set of algebraic extensions of $H$ in $F_n$.
\end{dfn}

\begin{thm}[Takahasi, \cite{T}; see also~\cite{MVW}]\label{takahasi}
Every $H\leqslant_{fg} F_n$ has finitely many algebraic extensions, say $\alge_{F_n}(H)=\{H=H_0, H_1, \ldots ,H_r\}$ ($r$ depending on $H$), each $H_i$ is finitely generated, and free bases for all of them are algorithmically computable from a given set of generators for $H$. Furthermore, for every extension $H\leqslant K\leqslant F_n$, there exists a unique (and computable) $0\leqslant i\leqslant r$ such that $H\leqslant_{alg} H_i \leqslant_{ff} K$; this $H_i$ is called the \emph{$K$-algebraic closure of $H$}.
\end{thm}

\begin{proof}[Sketch of the proof]
The original proof by M. Takahasi~\cite{T} was combinatorial, playing with words and cancellation in the free group. We sketch the modern proof given in~\cite{MVW} following ideas of Ventura~\cite{V}, Kapovich--Miasnikov~\cite{KM} and Margolis--Sapir--Weil~\cite{MSW}.

We have the alphabet $X$ fixed as a free basis for the ambient free group, $F_n=F(X)$. Now, given generators for $H\leqslant_{fg} F(X)$, one can compute the Stallings' graph $\Gamma(H)$ for $H$ (denote the basepoint by $\bp$). Attaching the necessary infinite hanging trees so that it becomes a complete graph (i.e., with all vertices having an incoming and an outgoing edge labelled $x_i$ for every $x_i\in X$), we obtain the Schreier graph $\chi(F_n, H, X)$ (which is finite if and only if $H$ is of finite index in $F_n$). Of course, $\chi(F_n, H, X)$ is a covering, $\chi(F_n, H, X)\twoheadrightarrow R(X)$, of the bouquet $R(X)$, the graph with a single vertex and one loop labelled $x_i$ for every $x_i\in X$; more precisely, it is the covering of $R(X)$ corresponding to the subgroup $H\leqslant_{fg} \pi(R(X))=F_n$. By standard covering theory, $K\leftrightarrow \chi(F_n, K, X)$ is a bijection between intermediate subgroups $H\leqslant K\leqslant F_n$ and intermediate coverings, $\chi(F_n, H, X)\twoheadrightarrow \chi(F_n, K, X) \twoheadrightarrow R(X)$ (mapping finitely generated subgroups to graphs with finite core, and viceversa).

Fix $H\leqslant_{fg} K\leqslant_{fg} F_n$, and consider their Stallings' graphs $\Gamma(H)=\core(\chi(F_n, H, X))$ and $\Gamma(K)=\core( \chi(F_n,K,X))$, both being finite graphs. The above bijection means that $\chi(F_n, K,X)$ is a quotient of $\chi(F_n, H,X)$, i.e., the result of $\chi(F_n, H,X)$ after identifying vertices and edges in a compatible way (i.e., modulo a \emph{congruence}, an equivalence relation on the set of vertices satisfying that if $p\sim q$ and $e_1$ and $e_2$ are edges with the same label and $\iota e_1=\iota e_2=p$, then $e_1\sim e_2$). There are two cases: if no pair of vertices in $\Gamma(H)\leqslant \chi(F_n, H,X)$ become identified then $\Gamma(H)$ is a subgraph of $\Gamma(K)=\core(\chi(F_n, K,X))$ and so, $H\leqslant_{ff} K$; otherwise, we loose $H$ from the picture, but we can still say that some compatible quotient of $\Gamma(H)$ will be visible as a subgraph of $\Gamma(K)$. Since $\Gamma(H)$ is finite, it has finitely many compatible quotients and, therefore, computing all of them and computing free bases for their fundamental groups, we obtain a finite list of finitely generated subgroups $\mathcal{O}_{F_n}(H)=\{ H=H_0, H_1, \ldots ,H_r\}$ ($r$ depending on $H$), called the \emph{fringe of $H$} in~\cite{MVW}, all of them containing $H$ and satisfying the following property: \emph{for every $H\leqslant_{fg} K\leqslant_{fg} F_n$ there exists (a non necessarily unique) $i=0,\ldots ,r$ such that $H\leqslant H_i\leqslant_{ff} K$}.

It only remains to clean this list by checking, for each pair of indices $i,j$, whether $H_i\leqslant_{ff} H_j$ and, in this case, delete $H_j$ from the list. It is not difficult to see that the resulting reduced list is precisely $\alge_{F_n}(H)\subseteq \mathcal{O}_{F_n}(H)$. Uniqueness of the $K$-algebraic closure follows directly from the definition of algebraic extension.
\end{proof}

As an easy corollary, we obtain the following result which immediately proves Theorem~\ref{main1}:

\begin{cor}
For any subgroup $H\leqslant_{fg}F_n$, we have $\dc_{F_n}(H)=\sup_{H\leqslant K\leqslant_{fg} F_n} \{ \rrk(H)/\rrk(K)\}=\max_{K\in \alge_{F_n}(H)} \{ \rrk(H)/\rrk(K)\}$; furthermore, we can effectively compute $\dc_{F_n}(H)$ and a free basis of a subgroup $K$ where the maximum is attained.
\end{cor}

\begin{proof}
By Theorem~\ref{takahasi}, every $H\leqslant K\leqslant_{fg} F_n$ uniquely determines the $K$-algebraic closure of $H$, i.e., an $H'\in \alge_{F_n}(H)$ such that $H\leqslant_{alg} H'\leqslant_{ff} K$. Since $\rrk(H')\leqslant \rrk(K)$, we can restrict the supremum in the definition of $\dc_{F_n}(H)$ to those subgroups in $\alge_{F_n}(H)$. And, since $|\alge_{F_n}(H)|$ is finite and computable, this supremum is a maximum and we can effectively compute both $\dc_{F_n}(H)$ and a free basis of a subgroup $K$ where the maximum is attained.
\end{proof}

\section{Degree of compression in free-abelian times free groups}\label{fatf-c}

For the rest of the paper we work in \emph{free-abelian times free} groups $G=\Z^m \times F_n$, i.e., direct products of a free-abelian group $\Z^m$ and a free group $F_n$, investigating here the degrees of compression and inertia of subgroups.

Taking a free-abelian basis $\{t_1,\ldots, t_m \}$ for $\Z^m$ (with multiplicative notation), and a free basis $\{ x_1,\ldots, x_n \}$ for $F_n$, we have
 $$
G=\Z^m \times F_n =\langle t_1, \ldots ,t_m, x_1,\ldots ,x_n \mid t_it_j = t_jt_i,\,\, t_ix_k = x_kt_i \rangle,
$$
where $i,j=1,2,\ldots ,m$ and $k=1,2,\ldots, n$. A normal form for elements in $G$ is
 $$
t^{a_1}_1 \cdots t^{a_m}_m u(x_1,\ldots,x_n) =t^{(a_1, \ldots , a_m)}u(x_1,\ldots,x_n),
 $$
where $a=(a_1, \ldots , a_m)\in \Z^m$ is a row integral vector, and $u=u(x_1,\ldots , x_n)$ is a reduced word in $F_n$. Note that the symbol $t$ by itself has no real meaning; it just allows us to convert the ambient notation for the abelian group $\Z^m$ from multiplicative into additive (since $t^a t^b=t^{a+b}$, for $a,b\in \Z^m$).

At a first glance, solving problems in $\Z^m \times F_n$ seems to be reducible to the corresponding problems in $\Z^m$ and $F_n$. However, this is not always the case and many naive looking questions are much more complicated to answer in $\Z^m \times F_n$, rather than in $\Z^m$ and $F_n$. This is the case, for example, with the Howson property: both $\Z^m$ and $F_n$ are Howson but, as we saw above, $G=\Z^m \times F_n$ is not (as soon as $m\geqslant 1$ and $n\geqslant 2$).

Let $\pi\colon G=\Z^m \times F_n \twoheadrightarrow F_n$, $t^au \mapsto u$, be the natural projection. For a subgroup $H\leqslant_{fg}G$, a \emph{basis of $H$} is a set of generators for $H$ of the form $\{ t^{a_1}u_1, t^{a_2}u_2, \ldots, t^{a_r}u_r, t^{b_1}, t^{b_2}, \ldots, t^{b_s} \}$, where $\{ u_1,\ldots ,u_{r}\}$ is a free basis of $H\pi$, $a_i\in \Z^m$ for $i=1,2,\ldots, r$, and $\{ b_1,\ldots ,b_s\}$ is a free-abelian basis for $L_H =H\cap \Z^m$ (to avoid confusions, we will maintain the full names, \emph{free-abelian basis}, \emph{free basis}, and just \emph{basis}, to refer to $\Z^m$, $F_n$, and $G=\Z^m\times F_n$, respectively). According to~\cite[Prop.~1.9]{DeV}, every subgroup $H\leqslant_{fg} G$ admits a basis, computable from any given set of generators. Observe also that a subgroup $H\leqslant G$ is finitely generated if and only if $H\pi\leqslant F_n$ is so.

In this section we study the degree of compression of a given subgroup $H\leqslant_{fg} G$. The first lemma says that it is enough to consider those overgroups $K$ such that $H\pi\leqslant_{alg} K\pi$.

\begin{lemma}\label{2.9}
Let $H{\leqslant}_{fg} G=\Z^m\times F_n$. Then,
 $$
\dc_G (H):=\sup_{H\leqslant K\leqslant_{fg}G} \left\{ \frac{\rrk(H)}{\rrk(K)} \right\}=\max_{\scriptsize \begin{array}{c} H\leqslant K\leqslant_{fg} G \\ H\pi \leqslant_{alg} K\pi \end{array}} \left\{ \frac{\rrk(H)}{\rrk(K)} \right\}.
 $$
\end{lemma}

\begin{proof}
We already observed above that the supremum defining the degree of compression is always a maximum. The inequality $\geqslant$ is clear.

Fix a basis for $H$, say $\{ t^{a_1}u_1, \ldots, t^{a_r}u_r, t^{b_1}, \ldots, t^{b_s} \}$. To see the other inequality, take a subgroup $H\leqslant K\leqslant_{fg}G$ and we shall construct $H\leqslant K'\leqslant_{fg}G$ such that $H\pi\leqslant_{alg} K'\pi$ and $\rrk(H)/\rrk(K)\leqslant \rrk(H)/\rrk(K')$.

We have $L_H=H\cap \Z^m =\langle t^{b_1}, \ldots ,t^{b_s}\rangle\leqslant K\cap \Z^m =L_K$ and $H\pi\leqslant K\pi$ so, $\rk(L_H) \leqslant \rk(L_K)$ and $H\pi \leqslant_{alg} J\leqslant_{ff} K\pi$, for some $J\in \alge_{F_n}(H\pi)$. Take a free basis $\{ v_1, \ldots, v_p\}$ for $J$ and extend it to a free basis $\{ v_1, \ldots, v_p,v_{p+1},\ldots, v_q\}$ for $K\pi$, $p\leqslant q$. Now, consider a basis for $K$ of the form $\{ t^{c_1}v_1, \ldots, t^{c_p}v_p, t^{c_{p+1}}v_{p+1}, \ldots ,t^{c_q}v_q, t^{d_1}, \ldots ,t^{d_{\ell }} \}$, where $c_i\in \Z^m$, $i=1,\ldots ,q$, are certain vectors, and $\{ t^{d_1}, \ldots, t^{d_{\ell }}\}$ is a free-abelian basis for $L_K$.

Let, $K'=\langle t^{c_1}v_1, \ldots, t^{c_p}v_p, t^{d_1}, \ldots, t^{d_{\ell }}\rangle\leqslant_{fg}K\leqslant G$ and we claim that $H\leqslant K'$. In fact, we already know that $t^{b_i}\in L_H\leqslant L_K=L_{K'} =\langle t^{d_1}, \ldots, t^{d_{\ell }}\rangle \leqslant K'$ for $i=1,\ldots ,s$. Now, for $i=1,\ldots ,r$ we see that $t^{a_i} u_i\in K'$: write $u_i$ as a word $u_i=w_i(v_1,\ldots, v_p)$ (unique up to reduction) and compute $w_i(t^{c_1}v_1,\ldots,t^{c_p}v_p)= t^{e_i}w_i(v_1,\ldots,v_p)=t^{e_i}u_i \in K'\leqslant K$, where $e_i=|w_i|_{v_1}c_1+\cdots +|w_i|_{v_p}c_p$. But $t^{a_i}u_i\in H\leqslant K$ so, $t^{e_i-a_i}\in L_K =L_{K'}\leqslant K'$ and hence, $t^{a_i}u_i=(t^{e_i-a_i})^{-1}(t^{e_i}u_i)\in K'$.

So, for every $H\leqslant K\leqslant_{fg} G$ we have found a finitely generated subgroup in between, $H\leqslant K'\leqslant K$, such that $H\pi \leqslant_{alg} J=K'\pi$ and
 $$
\rrk(K') =\rrk(K'\pi) +\rk(L_{K'} )=(p-1) +\rk(L_{K'})\leqslant (q-1)+\rk(L_K)=\rrk(K);
 $$
therefore, $\rrk(H)/\rrk(K') \geqslant \rrk(H)/\rrk(K)$ and the proof is completed.
\end{proof}

Fix $H\leqslant_{fg} G$ together with a basis for it $\{ t^{a_1}u_1, \ldots, t^{a_r}u_r, t^{b_1}, \ldots, t^{b_s}\}$, and consider the matrices
 $$
A=\left(\begin{array}{c} a_1 \\ \vdots \\ a_r \end{array}\right)\in M_{r \times m}(\Z)\quad \text{and} \quad B= \left(\begin{array}{c} b_1 \\ \vdots \\ b_s \end{array}\right )\in M_{s \times m}(\Z).
 $$
For every $J\in \alge_{F_n}(H\pi)$ given with a free basis, say $J=\langle v_1, \ldots, v_p\rangle$, we can consider the (unique reduced) word expressing each $u_i$ in terms of $v_1, \ldots, v_p$, say $u_i =w_i(v_1,\ldots, v_p)$, abelianize, and get the vector $({|w_i|}_{v_1},\ldots ,{|w_i|}_{v_p})\in \Z^p$, $i=1,\ldots ,r$; collecting all of them into the rows of a matrix,
 $$
U_J=\left (\,\begin{array}{ccc}
{|w_1|}_{v_1} & \cdots & |w_1|_{v_p}\\ & \vdots & \\ {|w_r|}_{v_1} & \cdots & |w_r|_{v_p} \end{array}\right) \in M_{r\times p}(\Z).
 $$

According to Lemma~\ref{2.9}, to compute $\dc_G(H)$ it is enough to consider the subgroups of the form $K=\langle t^{c_1}v_1, \ldots, t^{c_p}v_p, L_K\rangle\leqslant_{fg} G$ (where $L_K=K\cap \Z^m$, assume the given set of generators to be a basis for $K$) such that $H\leqslant K\leqslant G$, $H\pi =\langle u_1, \ldots, u_r \rangle \leqslant_{alg} K\pi =\langle v_1, \ldots, v_p \rangle$, compute $\rrk(H)/\rrk(K)$, and take the maximum of these values. (Observe that, although $|\alge_{F_n}(H\pi)|<\infty$, there are, possibly, infinitely many such $K$'s; however, $\rrk(K)=p-1+\rk(L_K)$ takes only finitely many values.)

So, fix such a $K$ and consider the matrix
 $$
C_K=\left( \begin{array}{c} c_1 \\ \vdots \\ c_p \end{array}\right) \in M_{p\times m}(\Z).
 $$
Observe that $C_K$ satisfies $\row(A-U_{K\pi}C_K)\leqslant L_K$: in fact, for every $i=1,\ldots ,r$, we have
 $$
K\ni w_i(t^{c_1}v_1,\ldots, t^{c_p}v_p)=t^{|w_i|_{v_1}c_1+\cdots +|w_i|_{v_p}c_p}w_i(v_1,\ldots, v_p)=t^{(U_{K\pi})_i C_K}u_i,
 $$
where $(U_{K\pi})_i$ is the $i$-th row of $U_{K\pi}$; therefore, $H\leqslant K$ implies that $a_i-(U_{K\pi})_i C_K\in L_K$, for $i=1,\ldots ,r$. This motivates the following definition, which allows us to obtain the main result in this section.

\begin{dfn}
For given matrices $A\in M_{r \times m}(\Z)$, $B\in M_{s \times m}(\Z)$, and $U\in M_{r \times p}(\Z)$, define $d(A,B,U)=\min_{L\leqslant \Z^m} \{ \rk(L) \mid \exists \,\, C\in M_{p\times m}(\Z) \text{ such that } \row(A-UC)\leqslant L, \text{ and } \row(B)\leqslant L \}$.
\end{dfn}


\begin{proof}[Proof of Theorem~\ref{main-C}]
By Lemma~\ref{2.9}, we know that the supremum in $\dc_G(H)$ is attained at a certain $H\leqslant K\leqslant_{fg}G$ such that $K\pi\in \mathcal{AE}_{F_n}(H\pi)$. And, for every such $K$, $\rrk(K)=\rrk(K\pi)+\rk(L_K)$ so,
 $$
\dc_G(H)=\max_{\scriptsize \begin{array}{c} H\leqslant K\leqslant_{fg} G \\ H\pi \leqslant_{alg} K\pi \end{array}} \left\{ \frac{\rrk(H)}{\rrk(K)} \right\} =\max_{J\in \mathcal{AE}_{F_n}(H\pi)} \left\{ \frac{\rrk(H)}{\rrk(J)+d(A,B,U_J)} \right\} =
 $$
 \begin{equation}\label{aa}
=\frac{\rrk(H)}{\min_{J\in \mathcal{AE}_{F_n}(H\pi)} \{ \rrk(J)+d(A,B,U_J)\}}
 \end{equation}
since, by the argument above, every $K$ with $K\pi=J\in \mathcal{AE}_{F_n}(H\pi)$ satisfies $\rk(L_K)\geqslant d(A,B,U_J)$, one of them with equality.

In order to compute the value of $\dc_G(H)$ we can do the following: first compute $\mathcal{AE}_{F_n}(H\pi)$; for each member $J=\langle v_1,\ldots ,v_p\rangle$, write each $u_i$ in the free basis of $H\pi$ in terms of the free basis $\{ v_1, \ldots ,v_p\}$ of $J$, and obtain the matrix $U_J$; then compute $d(A,B,U_J)+\rrk(J)$ (which is effectively doable by the following Proposition~\ref{2.11}). When this procedure is done for each of the finitely many $J\in \mathcal{AE}_{F_n}(H\pi)$, take the minimum of the values $d(A,B,U_J)+\rrk(J)$ and, by~\eqref{aa}, we are done. Moreover, the elements of the free basis for the subgroup $J$ attaining this minimum, together with the rows of the matrix $C$ just computed and realizing the minimum in $d(A,B,U_J)$, are the ingredients to build a basis for a subgroup $K\leqslant_{fg} G$ attaining the minimum in $\dc_G(H)$.
\end{proof}


\begin{prop}\label{2.11}
For any given matrices $A\in M_{r\times m}(\Z)$, $B\in M_{s\times m}(\Z)$, and $U\in M_{r\times p}(\Z)$, the value of $d(A,B,U)$ is algorithmically computable, together with a free-abelian basis of an $L\leqslant \Z^m$ attaining the minimum, and the corresponding matrix $C\in M_{p\times m}(\Z)$.
\end{prop}

\begin{proof}
Recall that $d(A,B,U)$ is the minimum rank of those subgroups $L\leqslant \Z^m$ satisfying $\row(B)\leqslant L$, and $\row(A-UC) \leqslant L$ for some $C\in M_{p\times m}(\Z)$. Observe first that, replacing $B$ by $B'$ with $\row(B)\leqslant_{fi} \row(B') \leqslant_{\oplus}\Z^m$, we have $d(A,B',U)=d(A,B,U)$; in fact, $d(A,B',U)\geqslant d(A,B,U)$ is clear from the definition, and for every $L\leqslant \Z^m$ containing $\row(B)$ and $\row(A-UC)$ for some $C\in M_{p\times m}(\Z)$, we have the subgroup $L+\row(B')\leqslant \Z^m$ which contains $\row(B')$ and $\row(A-UC)$ for the same matrix $C$, and has the same rank, $\rk(L+\row(B'))=\rk(L)$, since $L\leqslant_{fi} L+\row(B')$; this proves the equality.

Let us do a few reductions to the problem. Compute matrices $P\in \GL_r(\Z)$, $Q\in \GL_p(\Z)$, and positive integers $d_1,\ldots ,d_{\ell} \in \N$, $\ell \leqslant \min\{r,p\}$, satisfying $1\leqslant d_1 |d_2 |\cdots |d_{\ell}\neq 0$, such that $PUQ=U'$,  where $U'=\diag(d_1, \ldots ,d_{\ell})\in M_{r\times p}(\Z)$ (understanding the last $r-\ell\geqslant 0$ rows and the last $p-\ell\geqslant 0$ columns full of zeros); this is the Smith normal form of $U$, see~\cite{artin} for details. Writing $A'=PA$, $B'=B$, and doing the change of variable $C=QC'$, we have $\row(A-UC)= \row(PA-PUQC')=\row(A'-U'C')$. So, $d(A,B,U)=d(A',B',U')$.

To compute $d(A',B',U')$, we have to find a subgroup $L\leqslant \Z^m$ of the minimum possible rank, and vectors $c'_1, \ldots ,c'_p\in \Z^m$, such that $\row(B')\leqslant L$,
 \begin{equation}\label{cond1}
\left. \begin{array}{c}
a'_1-d_1 c'_1\in L \\ \cdots \\ a'_{\ell}-d_{\ell}c'_{\ell}\in L
\end{array} \right\},
\end{equation}
and
\begin{equation}\label{cond2}
\left. \begin{array}{c}
a'_{\ell+1}\in L \\ \cdots \\ a'_{r}\in L
\end{array} \right\}.
\end{equation}
Note that the last $p-\ell\geqslant 0$ columns of $U'$ are full of zeroes so, no condition concerns the vectors $c'_{\ell+1}, \ldots ,c'_p$ and we can take them to be arbitrary (say zero, for example). That is, taking $c'_{\ell+1}=\cdots =c'_p=0$, denoting $A''=A'\in M_{r\times m}(\Z)$, $B''=B'\in M_{s\times m}(\Z)$, $U''\in M_{r\times \ell}(\Z)$ the matrix $U'$ after deleting the last $p-\ell\geqslant 0$ columns (and $C''\in M_{\ell\times m}(\Z)$ the matrix $C'$ after deleting the last $p-\ell\geqslant 0$ rows), we have $d(A',B',U')=d(A'',B'',U'')$.

Now, we can ignore conditions~\eqref{cond2} by adding the vectors $a''_{\ell+1}, \ldots ,a''_r$ as extra rows at the bottom of $B$: let $A'''\in M_{\ell\times m}(\Z)$ be $A''$ after deleting the last $r-\ell\geqslant 0$ rows, $B'''\in M_{(s+r-\ell)\times m}(\Z)$ be $B''$ enlarged with $r-\ell$ extra rows with the vectors $a''_{\ell+1}, \ldots ,a''_r$, $U'''\in M_{\ell \times \ell}(\Z)$ be the matrix $U''$ after deleting the last $r-\ell\geqslant 0$ rows (and $C'''=C''$), and we have that $d(A'',B'',U'')=d(A''',B''',U''')$. Note that now, $U'''=\diag(d_1, \ldots ,d_{\ell})$ is a square matrix.

Finally, if $d_1=1$ we can take $c'_1=a'_1$ and the first condition in~\eqref{cond1} becomes trivial; so, deleting the possible ones at the beginning of the list $d_1|d_2|\cdots |d_{\ell}$ (and their rows and columns from $U'''$), and deleting also the corresponding first rows of $A$ and $C$, we can assume $d_1\neq 1$.

Altogether, and resetting the notation to the original one, we are reduced to compute $d(A,B,U)$ in the special situation where $A\in M_{r\times m}(\Z)$, $B\in M_{s\times m}(\Z)$, and $U=\diag(d_1, \ldots , d_{r})\in M_{r\times r}(\Z)$, with $1\neq d_1|d_2|\cdots |d_r\neq 0$, and further, by the argument in the first paragraph of the present proof, with $\row(B)$ being a direct summand of $\Z^m$. That is, we have to compute a subgroup $L\leqslant \Z^m$ of the minimum possible rank, and vectors $c_1, \ldots ,c_p\in \Z^m$ satisfying $\row(B)\leqslant L$ and
\begin{equation}\label{cond3}
\left. \begin{array}{c}
a_1-d_1 c_1\in L \\ \cdots \\ a_r-d_rc_r\in L
\end{array} \right\},
\end{equation}
where $a_i$ is the $i$-th row of $A$. Let us think the conditions in~\eqref{cond3} as saying that $a_i\in L$ modulo $d_i\Z^m$, $i=1,\ldots ,r$. To solve this, let us start with $L_0=\row(B)\leqslant_{\oplus} \Z^m$ and let us increase it the minimum possible in order to fulfill conditions~\eqref{cond3}.

Since $d_1 |d_2 |\cdots |d_r$, the natural projections $\pi_i \colon \Z^m \twoheadrightarrow (\Z/d_i\Z)^m$ factorize through the chain of morphisms $\Z^m \twoheadrightarrow (\Z/d_r\Z)^m \twoheadrightarrow (\Z/d_{r-1}\Z)^m \twoheadrightarrow \cdots \twoheadrightarrow (\Z/d_1\Z)^m$. Starting with $L\geqslant L_0$ and collecting the last condition in~\eqref{cond3}, we deduce that $L$ must further satisfy $L\pi_r\geqslant L_0\pi_r+\langle v^0_r\pi_r\rangle$, where $v^0_r= a_r\in \Z^m$. Now the second condition from below in~\eqref{cond3} adds the requirement $L\pi_{r-1}\ni a_{r-1}\pi_{r-1}$. But $a_{r-1}\pi_{r-1}\in (\Z/d_{r-1}\Z)^m$ has finitely many (more precisely, $(d_r/d_{r-1})^m$) pre-images in $(\Z/d_r\Z)^m$; compute them all, take pre-images $v_{r-1}$ up in $\Z^m$, and we get that $L$ must further satisfy $L\pi_r\geqslant L_0\pi_r+\langle v^0_r\pi_r,\, v_{r-1}\pi_r\rangle$, where $v_{r-1}\pi_r$ is one of these $(d_r/d_{r-1})^m$ pre-images. Repeat this same argument with all the conditions in~\eqref{cond3}, working from bottom to top: we deduce that $L$ must further satisfy $L\pi_r\geqslant L_0\pi_r+\langle v^0_r\pi_r,\, v_{r-1}\pi_r,\ldots , v_1\pi_r\rangle$, where $v_i\in \Z^m$ is a vector such that $v_i\pi_r$ is one of the computed $(d_r/d_i)^m$ pre-images of $a_i\pi_i\in (\Z/d_i\Z)^m$ up in $(\Z/d_r\Z)^m$, $i=r-1, \ldots ,1$, i.e., $v_i\equiv a_i \mod d_i$. This makes a total of $(d_r/d_{r-1})^m \cdots (d_r/d_1)^m$ possible lower bounds for $L\pi_r$: compute them all, find one with minimal possible rank, say $L\pi_r\geqslant L_0\pi_r+\langle v^0_r\pi_r,\, v^0_{r-1}\pi_r,\ldots , v^0_1\pi_r\rangle$, and we deduce that $d(A,U,B)\geqslant \rk(L_1\pi_r)$, where $L_1=L_0+\langle v^0_r,\, v^0_{r-1},\ldots , v^0_1\rangle\leqslant \Z^m$.

We claim that this lower bound is tight, i.e., $d(A,B,U)=\rk(L_1 \pi_r)$. To see this, we have to construct a subgroup $L\leqslant \Z^m$ of rank exactly $\rk(L_1\pi_r)$, containing $L_0$ and satisfying~\eqref{cond3} for some vectors $c_1,\ldots ,c_r\in \Z^m$ (which must also be computed). Since $L_0$ is a direct summand of $\Z^m$, say with free-abelian basis $\{w_1, \ldots ,w_k\}$, we deduce that $L_0\pi_r$ is a direct summand of $(\Z/d_r\Z)^m$ with abelian basis $\{w_1\pi_r, \ldots ,w_k\pi_r\}$. So, $L_0\pi_r$ is also a direct summand of $L_1\pi_r\leqslant (\Z/d_r\Z)^m$; compute a complement and get vectors $v'_1, \ldots ,v'_l\in \Z^m$, $l\leqslant r$, such that $\{w_1\pi_r, \ldots ,w_k\pi_r, v'_1\pi_r, \ldots ,v'_l\pi_r\}$ is an abelian basis of $L_1\pi_r =L_0\pi_r \oplus V$; in particular, $\rk(L_1\pi_r)=k+l$.

Finally, take $L=\langle w_1, \ldots ,w_k, v'_1, \ldots ,v'_l\rangle\leqslant \Z^m$. This subgroup has the desired rank $\rk(L)=k+l=\rk(L_1\pi_r)$ (since the given generators are linearly independent because their $\pi_r$-projections are so), and satisfies the required conditions: on one hand, $L_0=\langle w_1, \ldots ,w_k\rangle\leqslant L$; on the other hand, for every $i=1,\ldots ,r$, $v^0_i\pi_r\in L_1\pi_r=\langle w_1\pi_r, \ldots ,w_k\pi_r\rangle \oplus \langle v'_1\pi_r, \ldots ,v'_l\pi_r\rangle$ so,
 \begin{align*}
v^0_i\pi_r & = \lambda_1(w_1\pi_r)+\cdots +\lambda_k(w_k\pi_r) +\mu_1(v'_1\pi_r) +\cdots +\mu_l(v'_l\pi_r) \\ & = (\lambda_1 w_1+\cdots +\lambda_k w_k +\mu_1 v'_1 +\cdots +\mu_l v'_l )\pi_r,
 \end{align*}
for some integers $\lambda_1, \ldots ,\lambda_k, \mu_1, \ldots ,\mu_l\in \Z$; thus, $L$ contains the vector $c_i=\lambda_1 w_1+\cdots +\lambda_k w_k +\mu_1 v'_1 +\cdots +\mu_l v'_l$ which satisfies $c_i\equiv v^0_i \mod d_r$ and so, $c_i\equiv v^0_i \mod d_i$ too; since $v^0_i \equiv a_i \mod d_i$, we deduce $c_i\equiv a_i \mod d_i$ and we are done.
\end{proof}

It is natural to ask whether the minimum $\min_{J\in \mathcal{AE}_{F_n}(H\pi)} \{\rrk(J)+d(A,B,U_J)\}$ in Theorem~\ref{main-C} is attained at an algebraic extension $J\in \mathcal{AE}_{F_n}(H\pi)$ of minimal rank. Unfortunately, this is not always the case, as shown in the following example. In order to compute $\dc_G(H)$, this fact forces us to run \emph{over all} algebraic extensions $J$ of $H\pi$, and compute $d(A,B,U_J)$ following the algorithm given in Proposition~\ref{2.11}, \emph{for each one}. We do not see any shortcut to this procedure, for the general case.

\begin{ex}
We exhibit an explicit example of a subgroup $H\leqslant_{fg}G$ having two algebraic extensions $J,J'\in \mathcal{AE}_{F_n}(H\pi)$ with $\rrk(J)<\rrk(J')$ but $\rrk(J)+d(A,B,U_J)> \rrk(J')+d(A,B,U_{J'})$.

Let $H=\langle t^{(-1,0)}b^2, t^{(1,0)}ac^{-1}ac^{-1}, t^{(0,1)}bac^{-1} \rangle \leqslant_{fg}G=\Z^2\times F_3$. Projecting, we have $H\pi=\langle b^2,ac^{-1}ac^{-1},bac^{-1}\rangle$, and Fig.~\ref{f1} represents the Stallings' graph $\Gamma_A(H\pi)$ for $H\pi$ as a subgroup of $F_3$ with respect to the ambient free basis $A=\{a,b,c\}$. Successively identifying pairs of vertices of $\Gamma_A(H\pi)$ and reducing the resulting $A$-labeled graph in all possible ways, one concludes that $\Gamma_A(H\pi)$ has nine congruences, whose corresponding quotient graphs are depicted in Figs.~\ref{f1} and~\ref{f2}; this is the fringe of $H\pi$, $\mathcal{O}_{F_n}(H\pi)$; see the proof of Theorem~\ref{takahasi} above.
 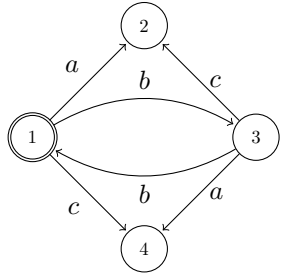
\begin{figure}
\begin{tikzpicture}[shorten >=1pt,node distance=3cm,auto]
\node[state,accepting, scale=0.70] (1) {$1$};
 \node[state, scale=0.70] (2) [above right of = 1] {$2$};
 \node[state, scale=0.70] (4) [below right of = 1] {$4$};
 \node[state, scale=0.70](3) [below right of = 2] {$3$};
 \path[->] (1) edge node {$a$} (2)
 edge node [swap] {$c$} (4)
 (3) edge node [right] {$c$} (2)
(3) edge node [right] [swap] {$a$}  (4);
 \path[->] (3) edge [bend left] node {$b$} (1);
 \path[->] (1) edge [bend left] node {$b$} (3);
\end{tikzpicture}
\caption{The Stallings' graph $\Gamma(H\pi)$}\label{f1}
\end{figure}

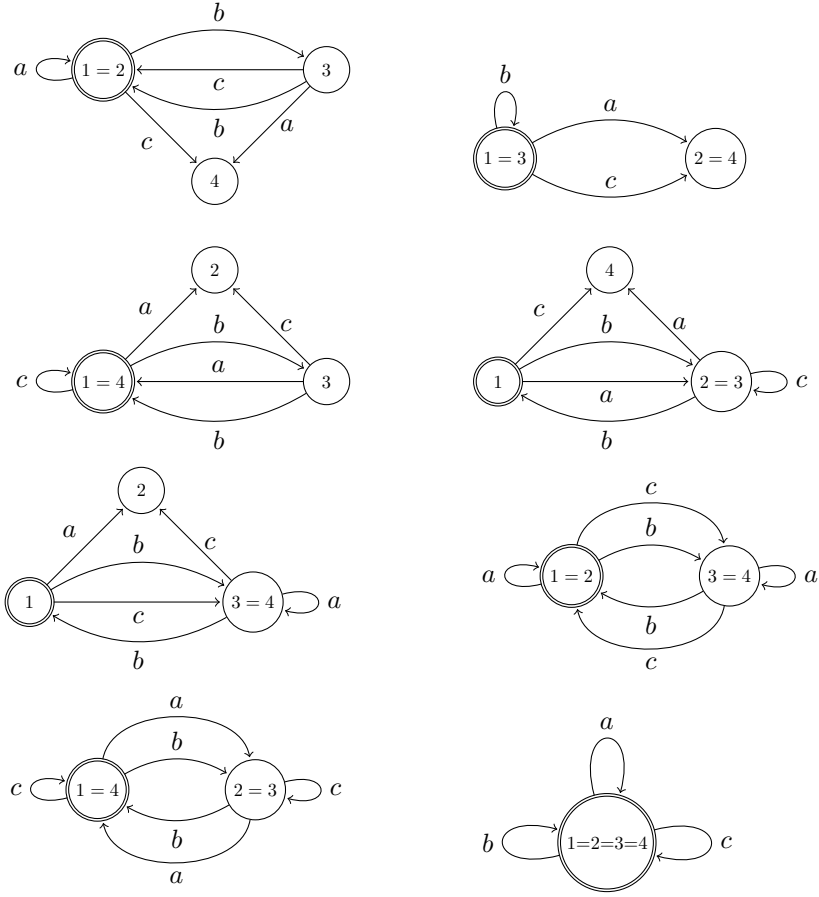
\begin{figure}
 $$
\begin{array}{rcl}
\begin{tikzpicture}[shorten >=1pt,node distance=3cm,auto]
\node[state,accepting, scale=0.70 ] (1) {$1=2$};
\node[state, scale=0.70] (4) [below right of = 1] {$4$};
\node[state, scale=0.70](3) [ below right of = 2] {$3$};
\path[->] (1) edge node [swap] {$c$} (4)
edge [loop left] node {$a$} ()
(3) edge node {$c$} (1)
(3) edge node [right] [swap] {$a$}  (4);
\path[->] (3) edge [bend left] node {$b$} (1);
\path[->] (1) edge [bend left] node {$b$} (3);
\end{tikzpicture}
& \phantom{aaaaa} &
\begin{tikzpicture}[shorten >=1pt,node distance=4cm,auto]
\node[state,accepting, scale=0.70] (1) {$1=3$};
\node[state, scale=0.70] (2) [right of = 1] {$2=4$};
\path [->] (1) edge [loop above] node {$b$} ();
\path[->] (1) edge [bend left] node {$a$} (2);
\path[->] (1) edge [bend right] node {$c$} (2);
\end{tikzpicture}
\\ & \\
\begin{tikzpicture}[shorten >=1pt,node distance=3cm,auto]
\node[state,accepting, scale=0.70] (1) {$1=4$};
\node[state, scale=0.70] (2) [above right of = 1] {$2$};
\node[state, scale=0.70](3) [below right of = 2] {$3$};
\path[->] (1) edge node {$a$} (2)
edge [loop left] node {$c$} ()
(3) edge node [right] {$c$} (2)
(3) edge node [swap] {$a$}  (1);
\path[->] (3) edge [bend left] node {$b$} (1);
\path[->] (1) edge [bend left] node {$b$} (3);
\end{tikzpicture}
& \phantom{aaaaa}  &
\begin{tikzpicture}[shorten >=1pt,node distance=3cm,auto]
\node[state,accepting, scale=0.70] (1) {$1$};
\node[state, scale=0.70] (2) [above right of = 1] {$4$};
\node[state,scale=0.70](3) [below right of = 2] {$2=3$};
\path[->] (1) edge node {$c$} (2)
(3) edge node [right] {$a$} (2)
(1) edge node [swap] {$a$}  (3);
\path[->] (3) edge [bend left] node {$b$} (1)
edge [loop right] node {$c$} ();
\path[->] (1) edge [bend left] node {$b$} (3);
 \end{tikzpicture}
\\
\begin{tikzpicture}[shorten >=1pt,node distance=3cm,auto]
\node[state,accepting, scale=0.70] (1) {$1$};
\node[state, scale=0.70] (2) [above right of = 1] {$2$};
\node[state,scale=0.70](3) [below right of = 2] {$3=4$};
\path[->] (1) edge node {$a$} (2)
(3) edge node [right] {$c$} (2)
(1) edge node [swap] {$c$}  (3);
\path[->] (3) edge [bend left] node {$b$} (1)
edge [loop right] node {$a$} ();
\path[->] (1) edge [bend left] node {$b$} (3);
 \end{tikzpicture}
& \phantom{aaaaa}  &
\begin{tikzpicture}[shorten >=1pt,node distance=3cm,auto]
\node[state,accepting, scale=0.70] (1) {$1=2$};
\node[state, scale=0.70] (2) [ right of = 1] {$3=4$};
\path[->] (1) edge [bend left=80] node {$c$} (2)
(2) edge [bend left=80] node {$c$} (1);
\path[->] (1) edge [bend left] node {$b$} (2)
edge [loop left] node {$a$} ();
\path[->] (2) edge [bend left] node {$b$} (1)
edge [loop right] node {$a$} ();
 \end{tikzpicture}
\\
\begin{tikzpicture}[shorten >=1pt,node distance=3cm,auto]
\node[state,accepting, scale=0.70] (1) {$1=4$};
\node[state, scale=0.70] (2) [ right of = 1] {$2=3$};
\path[->] (1) edge [bend left=80] node {$a$} (2)
(2) edge [bend left=80] node {$a$} (1);
\path[->] (1) edge [bend left] node {$b$} (2)
edge [loop left] node {$c$} ();
\path[->] (2) edge [bend left] node {$b$} (1)
edge [loop right] node {$c$} ();
 \end{tikzpicture}
 & \phantom{aaaaa}  &
\begin{tikzpicture}[shorten >=1pt,node distance=3cm,auto]
\node[state,accepting, scale=0.70] (1) {1=2=3=4};
\path [->] (1) edge [loop above] node {$a$} ();
\path [->] (1) edge [loop left] node {$b$} ();
\path [->] (1) edge [loop right] node {$c$} ();
\end{tikzpicture}
\end{array}
 $$
\caption{The eight non-trivial quotients of $\Gamma_A(H\pi)$}\label{f2}
\end{figure}

Now following the cleaning process, we get the set of algebraic extensions for $H\pi$, namely $\alge(H\pi)=\{H\pi, J\}$, where $J=\langle b, ac^{-1}\rangle\}$. (To this goal, the following fact helps: suppose $N$ is obtained from $M$ by a single identification of a pair of vertices followed by foldings; if $\rk(N)=\rk(M)+1$ then $M$ is a free factor of $N$, otherwise, $M\leqslant_{alg} N$.)

Following the notation above, we have
 $$
A=\left( \begin{array}{cc} -1 & 0 \\ 1 & 0 \\ 0 & 1 \end{array} \right), \qquad B=\emptyset, \qquad U_{H\pi}=\left( \begin{array}{ccc} 1 & 0 & 0 \\ 0 & 1 & 0 \\ 0 & 0 & 1 \end{array}\right), \quad U_{J}= \left( \begin{array}{cc} 2 & 0 \\ 0 & 2 \\ 1 & 1 \end{array}\right).
 $$
According to Theorem~\ref{main-C},
 \begin{equation}\label{min}
dc_G(H)=\rrk(H)/\min\{\rrk(H\pi)+d(A,B,U_{H\pi}),\, \rrk(J)+ d(A,B,U_J)\}.
 \end{equation}
Since $H\leqslant H$, $d(A,B,U_{H\pi})=\rk(L_H)=0$ and the first term on the minimum in~\eqref{min} is $\rrk(H\pi)+ d(A,B,U_{H\pi}) =(3-1)+0 =2$.

Following the algorithm given in Proposition~\ref{2.11}, let us compute now $d(A,B,U_J)$, where $J=\langle b, ac^{-1}\rangle$; we have $r=3$, $m=2$, $s=0$, and $p=2$. Computing the Smith normal form for $U_J$, we get
 $$
P=\left(\begin{array}{ccc} 0 & 0 & 1 \\ 0 & 1 & 0 \\ 1 & 1 & -2 \end{array}\right)\in \GL_3(\Z), \qquad Q=\left(\begin{array}{cc} 1 & -1 \\ 0 & 1 \end{array}\right)\in \GL_2(\Z), \qquad U'=\left( \begin{array}{cc} 1 & 0 \\ 0 & 2 \\ 0 & 0 \end{array}\right),
 $$
with $d_1=1$, $d_2=2$, and $\ell=\min\{r, p\}=2$. Diagonalyzing the problem, we obtain
 $$
A'=PA=\left(\begin{array}{cc} 0 & 1 \\ 1 & 0 \\ 0 & -2 \end{array} \right), \qquad B'=B=\emptyset, \qquad U'=\left(\begin{array}{cc} 1 & 0 \\ 0 & 2 \\ 0 & 0 \end{array}\right),
 $$
and $d(A,B,U_J)=d(A',B',U')$ (under the change of variable $C=QC'$). Since $p=\ell=2$ the next reduction is empty and $A''=A'$, $B''=B'$, and $U''=U'$. Applying the following reduction to delete the last $r-\ell=3-2=1$ zero rows in $U''$, we get
 $$
A'''=\left(\begin{array}{cc} 0 & 1 \\ 1 & 0 \end{array}\right), \qquad
B'''=\left(\begin{array}{cc} 0 & -2 \end{array}\right), \qquad U'''= \left(\begin{array}{cc} 1 & 0 \\ 0 & 2 \end{array}\right).
 $$
Finally, in order to delete $d_1=1$ from the list of divisors, we take $c'''_1=(0,1)$ and get
 $$
A''''=\left(\begin{array}{cc} 1 & 0 \end{array}\right), \qquad B''''= \left(\begin{array}{cc} 0 & -2 \end{array}\right), \qquad U''''= \left(\begin{array}{c} 2 \end{array}\right).
 $$
Going up by finite index, we replace the matrix $B''''$ to $(0,1)$, and are reduced to compute $d(A'''', (0,1), U'''')$; this is the smallest rank of a subgroup $L\leqslant \Z^2$ such that $\langle (0,1)\rangle\leqslant L$ and $(1,0)-2c''''_2\in L$ for some $c''''_2\in \Z^2$. Clearly, $d(A'''', (0,1), U'''')=2$, and one (non unique) solution is given by $L=\Z^2$ and $c''''_2=(1,0)$. Collecting the $c_1$ computed before, and undoing the change of variable, we get
 $$
C=QC'=QC'''=\left(\begin{array}{cc} 1 & -1 \\ 0 & 1 \end{array}\right) \left(\begin{array}{cc} 0 & 1 \\ 1 & 0 \end{array}\right) =\left(\begin{array}{cc} -1 & 1 \\ 1 & 0 \end{array}\right).
 $$
We conclude that $d(A,B,U_J)=2$ and one of the subgroups $K$ with the smallest possible rank satisfying $K\pi=J$ and $H\leqslant K\leqslant \Z^2\times F_3$ is $K=\langle t^{(-1,1)}b,\, t^{(1,0)}ac^{-1},\, t^{(1,0)},\, t^{(0,1)}\rangle$. So, the second term on the minimum in~\eqref{min} is $\rrk(J)+d(A,B,U_J)=(2-1)+2=3$. Therefore,
 \begin{align*}
dc_G(H)&=\frac{\rrk(H)}{\min\{\rrk(H\pi)+d(A,B,U_{H\pi}),\,\, \rrk(J)+d(A,B,U_J)\}} \\ &=\frac{3-1}{\min\{(3-1)+0,\, (2-1)+2)\}} \\ &=\frac{2}{2}=1.
 \end{align*}
In particular, $H$ is compressed in $G$.

As seen in this example, the algebraic extension $J$ looks better than the other one $H\pi$ because it contributes to the free rank in 2 units instead of 3. However, in order to match the free-abelian part, $J$ forces us to take two more units of rank, while $H\pi$ requires zero units. Note that in this example, $d(A,B,U_J)$ is as big as it could be since, in general, $d(A,B,U_J)\leqslant m=2$. The example can easily be extended to an arbitrary $m$.
\end{ex}

\section{Degree of inertia in free-abelian times free groups}\label{fatf-i}

In this section, we study the degree of inertia for subgroups $H$ of $G=\Z^m\times F_n$ and relate it to the corresponding degree of inertia of $H\pi$ in $F_n$; it turns out that the index of $H\cap \Z^m$ in $\Z^m$ (whether finite or infinite) is closely related to the degree of inertia of $H$. Unfortunately, the situation here is more complicated and we can only prove an upper bound for $\di_G(\cdot)$ in terms of $\di_{F_n}(\cdot)$ and the mentioned index; the computability of this function remains open, as in the free case.

\begin{lemma}\label{BadBoy}
For positive real numbers $a,b,c,d>0$,
 $$
\frac{a}{b}\leqslant \frac{c}{d} \,\, \Rightarrow \,\, \frac{a}{b}\leqslant \frac{a+c}{b+d}\leqslant \frac{c}{d}. \qed
 $$
\end{lemma}

\begin{proof}[Proof of Theorem~\ref{mainI}]
\emph{(i)}. The hypothesis $\rk(H\pi)\leqslant 1$ implies that $H=\langle t^a u, L_H\rangle$, for some $a\in \Z^m$ and $u\in F_n$ (possibly trivial). Then, for every $K\leqslant_{fg} G$, we have $(H\cap K)\pi\leqslant H\pi\cap K\pi\leqslant \langle u\rangle$ so, $(H\cap K)\pi =\langle u^r\rangle$ for some $r\in \Z$. Hence, $H\cap K=\langle t^bu^r, L_H\cap L_K\rangle$ for some $b\in \Z^m$ and we get  $\rk(H\cap K)\leqslant \rk(K)$. Therefore, $\rrk(H\cap K)/\rrk(K)\leqslant 1$, which is valid for every $K\leqslant_{fg} G$. Thus, $\di_G(H)=1$ (i.e., $H$ is inert in $G$).

\emph{(ii).} Consider the (unique) subgroup $\tilde{L}_H$ satisfying $L_H{\leqslant}_{fi} \tilde{L}_H{\leqslant}_{\oplus}\Z^m$, and take a free-abelian basis $\{ b_1, \ldots ,b_s\}$ of $\tilde{L}_H$, such that $\{ \lambda_1b_1,\ldots, \lambda_sb_s\}$ is a free-abelian basis of $L_H$ for appropriate integers $\lambda_1,\ldots, \lambda_s \in \Z$ (there is always a basis like this by standard linear algebra arguments). By hypothesis, $s=\rk(L_H)<m$ and, completing to a free-abelian basis $\{ b_1, \ldots ,b_s, b_{s+1}, \ldots ,b_m\}$ of the ambient $\Z^m$, we get at least one extra vector $b_{s+1}$ (which, of course, is primitive in $\Z^m$ and so has relatively prime coordinates).

Now fix a basis for $H$ of the form $\{ t^{a_1}u_1,\ldots,t^{a_{n_1}}u_{n_1}, t^{\lambda_1 b_1}, \ldots ,t^{\lambda_s b_s}\}$, where $a_1,\ldots,a_{n_1}\in \Z^m$, and $\{u_1,\ldots, u_{n_1}\}$ is a free basis for $H\pi$; in particular, we have $\rk(H\pi)=n_1\geqslant 2$, $\rk(L_H)=s<m$, and $\rk(H)=n_1+s$.

For proving $\di_G(H)=\infty$, we shall construct a family of subgroups $K_N\leqslant_{fg}\Z^m\times F_n$, indexed by $N\in \N$, all of them with constant rank 3 (i.e., $\rrk(K_N)=2$), with all the intersections $H\cap K_N$ being finitely generated, but with $\rrk(H\cap K_N)$ tending to $\infty$, as $N\to \infty$.

Let $K_N=\langle t^{{a'_1}}u_1, t^{{a'_2}}u_2, L_{K_N}\rangle\leqslant \Z^m\times F_n$, where the vectors $a'_1, a'_2\in \Z^m$ and the subgroup $L_{K_N}\leqslant \Z^m$ are to be determined; note that for all choices $\rk(K_N \pi)=2$, and here we are already using the hypothesis $n_1\geqslant 2$.

Let us understand the intersection $H\cap K_N$ following the procedure (and notation) given in~\cite[Thm.~4.5]{DeV}: we have $n_2=\rk(K_N\pi)=2$, $H\pi \cap K_N\pi=\langle u_1, u_2 \rangle$ and so $n_3=\rk(H\pi \cap K_N\pi)=2$, and we consider the matrices
 $$
A=\left (\begin{array}{c} a_1 \\ \vdots \\ a_{n_1} \end{array}\right ) \in M_{{n_1}\times m}(\Z),\quad  A'=\left (\begin{array}{c} a'_1 \\ a'_2 \end{array}\right )\in M_{2\times m}(\Z).
 $$
Let $\rho_1 \colon H\pi \twoheadrightarrow \Z^{n_1}$, $\rho_2 \colon K_N\pi \twoheadrightarrow \Z^2$, and $\rho_3 \colon H\pi \cap K_N\pi \twoheadrightarrow \Z^2$ be the corresponding abelianization maps (not to be confused with the restrictions of the global abelianization map $F_n \twoheadrightarrow \Z^n$ to the corresponding domains). Clearly, the inclusion maps $\iota_H\colon H\pi\cap K_N\pi \hookrightarrow H\pi$ and $\iota_K\colon H\pi\cap K_N\pi \hookrightarrow K_N\pi$ abelianize, respectively, to the morphisms $\Z^2\to \Z^{n_1}$ and $\Z^2\to \Z^2$ given by the matrices
 $$
P=\left( \begin{array}{ccccc}
1 & 0 & 0 & \hdots & 0 \\ 0 & 1 & 0 & \hdots & 0 \end{array}\right)\in M_{2\times n_1}(\Z), \quad P'=I_2=\left(\begin{array}{cc} 1 & 0 \\ 0 & 1 \end{array}\right) \in M_{2\times 2}(\Z).
 $$
Moreover, let
 $$
R=PA-P'A'=\left(\begin{array}{c} a_1 \\ a_2 \end{array} \right)-\left( \begin{array}{c} a'_1 \\ a'_2 \end{array}\right) = \left( \begin{array}{c} a_1-a'_1 \\ a_2-a'_2 \end{array} \right)\in M_{2\times m}(\Z),
 $$
and let us put all these ingredients into the following diagram:
 \begin{equation}\label{dia}
\begin{aligned}
\xy
(0,5)*+{\rotatebox[origin=c]{270}{$\leqslant$}};
(0,10)*{(H \cap K_N) \pi};
(0,0)*+{H \pi \cap K_N \pi}; (-25,0)*+{H \pi}; (25,0)*+{K_N \pi};
{\ar@{_(->}_-{\iota} (0,0)*++++++++++{}; (-25,0)*++++{}};
{\ar@{^(->}^-{\iota'} (0,0)*++++++++++{}; (25,0)*++++{}};
(0,-20)*+{\Z^{2}}; (-25,-20)*+{\ZZ^{n_1}}; (25,-20)*+{\ZZ^{2}};
{\ar@{->>}^-{\rho_3} (0,0)*+++{}; (0,-20)*+++{}};
{\ar@{->>}_-{\rho_1} (-25,0)*+++{}; (-25,-20)*+++{}};
{\ar@{->>}^-{\rho_2} (25,0)*+++{}; (25,-20)*+++{}};
(-12.5,-10)*+{///};
(12.5,-10)*+{///};
(0,-20)*+{\ZZ^{2}}; (-25,-20)*+{\ZZ^{n_1}}; (25,-20)*+{\ZZ^2};
{\ar_-{P} (0,-20)*+++++{}; (-25,-20)*++++{}};
{\ar^-{P'} (0,-20)*+++++{}; (25,-20)*++++{}};
(0,-40)*+{\ZZ^{m}};
{\ar^-{A} (-25,-20)*++++{}; (0,-40)*++++{}};
{\ar_-{A'} (25,-20)*++++{}; (0,-40)*++++{}};
{\ar
^-{R} (0,-20)*++++{}; (0,-40)*++++{}};
\endxy
\end{aligned}
 \end{equation}
According to the argument in~\cite[Thm.~4.5]{DeV}, the subgroup $(H\cap K_N)\pi \leqslant H\pi\cap K_N\pi$ is, precisely, the full preimage by $R$ and $\rho_3$ of $L_H+L_{K_N}\leqslant \Z^m$.

Let us choose now the vectors $a'_1=a_1-b_{s+1}$ and $a'_2=a_2$, and the subgroup $L_{K_N}=\langle Nb_{s+1}\rangle$, $N\in \mathbb{N}$; the matrix $R$ becomes
 $$
R=\left ( \, \begin{array}{c} b_{s+1} \\ 0 \end{array} \right).
 $$
We have $L_H+L_{K_N}=\langle \lambda_1b_1, \ldots ,\lambda_sb_s, Nb_{s+1}\rangle$ and then,
 $$
\begin{array}{ccl}
(L_H+L_{K_N})R^{-1} & = & \{(x,y)\in \Z^2 \mid (x\,\,\, y)R\in L_H+L_{K_N} \} \\ & = & \{(x,y)\in \Z^2 \mid xb_{s+1}\in L_H+L_{K_N} \} \\ & = & \{(x,y)\in \Z^2 \mid xb_{s+1}\in \langle Nb_{s+1}\rangle \} \\ & = & N\Z \times \Z \quad {\leqslant}_{N} \,\,\, \Z^2
\end{array}
 $$
(the last equality being true because $b_{s+1}$ has relatively prime coordinates). As $\rho_3$ is onto, taking $\rho_3$-preimages preserves the index and we have
 $$
(H\cap K_N)\pi=(L_H+L_{K_N})R^{-1}{\rho_3}^{-1}=(N\Z \times \Z){\rho_3}^{-1}\leqslant_N H\pi \cap K_N\pi.
 $$
Thus, by the Schreier index formula, $\rrk((H\cap K_N)\pi)=N\rrk(H\pi\cap K_N\pi)=N$ and we deduce that $\rrk(H\cap K_N)=N+\rk(L_H\cap L_{K_N})=N+0=N$ tends to $\infty$, as $N\to \infty$. This completes the proof that $\di_G(H)=\infty$.

\emph{(iii).} Fix a basis for $H$, say $\{ t^{a_1}u_1,\ldots, t^{a_{n_1}}u_{n_1}, t^{b_1}, \ldots ,t^{b_m}\}$, where $a_1,\ldots,a_{n_1}\in \Z^m$, $\{u_1,\ldots, u_{n_1}\}$ is a free basis for $H\pi$, and $\{ b_1,\ldots ,b_m\}$ is a free-abelian basis for $L_H\leqslant_{l} \Z^m$; in particular, $\rk(H\pi)=n_1\geqslant 2$, $\rk(L_H)=m$, and $\rk(H)=n_1+m$.

In order to show the inequality $\di_G(H)\leqslant l\di_{F_n}(H\pi)$, let us fix an arbitrary subgroup $K\leqslant_{fg} G$, assume that $H\cap K$ is finitely generated, and let us prove that $\rrk(H\cap K)/\rrk(K)\leqslant l\di_{F_n}(H\pi)$. Fix a basis for $K$, say $K=\langle t^{a_1'}v_1,\ldots, t^{a_{n_2}'}v_{n_2},L_K \rangle$ and we have
 \begin{equation}\label{10}
\frac{\rrk(H\cap K)}{\rrk(K)}=\frac{\rrk((H\cap K)\pi)+\rk(L_H\cap L_K)}{\rrk(K\pi)+\rk(L_K)}.
\end{equation}
As in the proof of part (ii), we consider the intersection diagram to understand $H\cap K$:
 \begin{equation}\label{m}
\begin{aligned}
\xy
(0,5)*+{\rotatebox[origin=c]{270}{$\leqslant$}};
(0,10)*{(H \cap K) \pi};
(0,0)*+{H \pi \cap K\pi}; (-25,0)*+{H \pi}; (25,0)*+{K\pi};
{\ar@{_(->}_-{\iota} (0,0)*++++++++++{}; (-25,0)*++++{}};
{\ar@{^(->}^-{\iota'} (0,0)*++++++++++{}; (25,0)*++++{}};
(0,-20)*+{\ZZ^{n_3}}; (-25,-20)*+{\ZZ^{n_1}}; (25,-20)*+{\ZZ^{n_2}};
{\ar@{->>}^-{\rho_3} (0,0)*+++{}; (0,-20)*+++{}};
{\ar@{->>}_-{\rho_1} (-25,0)*+++{}; (-25,-20)*+++{}};
{\ar@{->>}^-{\rho_2} (25,0)*+++{}; (25,-20)*+++{}};
(-12.5,-10)*+{///};
(12.5,-10)*+{///};
(0,-20)*+{\ZZ^{n_3}}; (-25,-20)*+{\ZZ^{n_1}}; (25,-20)*+{\ZZ^{n_2}};
{\ar_-{P} (0,-20)*+++++{}; (-25,-20)*++++{}};
{\ar^-{P'} (0,-20)*+++++{}; (25,-20)*++++{}};
(0,-40)*+{\ZZ^{m}};
{\ar^-{A} (-25,-20)*++++{}; (0,-40)*++++{}};
{\ar_-{A'} (25,-20)*++++{}; (0,-40)*++++{}};
{\ar
^-{R} (0,-20)*++++{}; (0,-40)*++++{}};
\endxy
\end{aligned}
\end{equation}
where $\rho_1 \colon H\pi \twoheadrightarrow \Z^{n_1}$, $\rho_2 \colon
K\pi \twoheadrightarrow \Z^{n_2}$, and $\rho_3 \colon H\pi \cap K\pi \twoheadrightarrow \Z^{n_3}$ are the corresponding abelianization maps (here, $n_3=\rk(H\pi\cap K\pi)<\infty$), where $\iota$ and $\iota'$ are the natural inclusions, where $P\in M_{n_3\times n_1}(\Z)$ and $P'\in M_{n_3\times n_2}(\Z)$ are the matrices of their respective abelianizations (note that $\iota$ and $\iota'$ being injective do not imply $P$ and $P'$ necessarily being so; in particular, $n_3$ may very well be bigger than $n_1$ or $n_2$), where $A\in M_{n_1\times m}(\Z)$ and $A'\in M_{n_2\times m}(\Z)$ are the matrices with rows $\{a_1,\ldots,a_{n_1}\}$ and $\{a_1',\ldots,a_{n_2}'\}$ respectively, and where $R=PA-P'A'\in M_{n_3\times m}(\Z)$. According to the argument in~\cite[Thm.~4.5]{DeV}, the crucial property of diagram~(\ref{m}) is the fact that $(H\cap K)\pi =(L_H+L_K)R^{-1}\rho_3^{-1}$.

From the hypothesis, $L_H \leqslant_l \Z^m$ and so, $L_H+L_K \leqslant_{l'}\Z^m$, where $1\leqslant l'\leqslant l$. As in general $R$ is not necessarily onto, $(L_H+L_K)R^{-1}\leqslant_{l''}\Z^{n_3}$ with $1\leqslant l''\leqslant l'$. And, since $\rho_3$ is onto, $(H\cap K)\pi =(L_H+L_K)R^{-1}\rho_3^{-1}\leqslant_{l''} H\pi \cap K\pi$. Therefore, by the Schreier index formula,
 \begin{equation}\label{lost1}
\begin{array}{ccccc}
\rrk((H \cap K)\pi)=l''\rrk(H\pi \cap K\pi) & = & l''\frac{\rrk(H\pi \cap K\pi)}{\rrk(K\pi)}\rrk(K\pi) & \leqslant & l''\di_{F_n}(H\pi) \rrk(K\pi).
\end{array}
 \end{equation}
Now, using~\eqref{10}, we have
 \begin{equation}\label{11}
\frac{\rrk(H \cap K)}{\rrk(K)} \leqslant \frac{l''\di_{F_n}(H\pi) \rrk(K\pi)+ \rk(L_H\cap L_K)}{\rrk(K\pi)+\rk(L_K)}\leqslant \frac{l''\di_{F_n}(H\pi)\rrk(K\pi)}{\rrk(K\pi)}=l''\di_{F_n}(H\pi),
 \end{equation}
where the second inequality is an equality if $L_K=\{0\}$, and follows from applying Lemma~\ref{BadBoy} to $\frac{\rk(L_H\cap L_K)}{\rk(L_K)}\leqslant 1 \leqslant l''\di_{F_n}(H\pi)$ otherwise. Therefore
 \begin{equation}\label{lost2}
\frac{\rrk(H\cap K)}{\rrk (K)}\leqslant l''\di_{F_n}(H\pi)\leqslant l'\di_{F_n}(H\pi)\leqslant l\di_{F_n}(H\pi),
 \end{equation}
as we wanted.
\end{proof}

\section{Restricted degree of inertia for free-abelian times free groups}\label{section for rdi}

To improve the inequality from Theorem~\ref{mainI}~(iii) into an equality, we need to add a couple of technical restrictions on the subgroups $K$ over which the supremum in the definition of degree of inertia runs. This gives rise to the notion of restricted degree of inertia given in Definition~\ref{res def di}; in the particular case of interest, $G=\Z^m\times F_n$, it is the following:
 $$
\di'_G(H) =\sup_{\tiny{\begin{array}{c} K\leqslant_{fg} G \\ H\cap K\leqslant_{fg} G \\ {[}H\pi : H\pi \cap K\pi{]} = \infty \\  H\pi \cap K\pi \nleqslant [F_n, F_n] \end{array} } } \left\{ \frac{\trk(H \cap K)}{\trk(K)} \right\} \leqslant \di_G(H),
 $$
applied to subgroups $H\leqslant_{fg} G$ such that $r(H\pi)\geqslant 2$ and $H\pi \nleqslant [F_n, F_n]$. The main result in the present section is Theorem~\ref{mainII}. The proofs for part (i) and for the inequality $\leqslant$ from (ii) work almost exactly equal as the corresponding parts from Theorem~\ref{mainI}. The inequality $\di_G(H) \geqslant l\di_{F_n}(H\pi)$ from (ii) is more involved and will require the previous development of several lemmas about intersections of subgroups of $F_n$, and a strong use of the well-known tool of pull-backs of graphs for working with intersections of finitely generated subgroups of $F_n$; we concentrate these technicalities into Claim~\ref{hard} and postpone its proof until having the lemmas available.

\begin{proof}[Proof of Theorem~\ref{mainII}]
\textit{(i)}. Follow the same arguments as in Theorem~\ref{mainI}~(ii) with the following detail in mind: by the assumption $H\pi\not\in [F_n, F_n]$ we can assume, from the very beginning and without loss of generality, that $u_1 \notin [F_n, F_n]$, i.e., the first element in the chosen free-basis for $H\pi$ is outside the commutator $[F_n, F_n]$. Now the goal is to construct a family of subgroups $K_N \leqslant_{fg} \Z^m\times F_n$, indexed by $N\in \N$, all of them having rank $3$, with all the intersections $H\cap K_N$ being finitely generated, \emph{and further satisfying} $[H\pi : H\pi \cap K_N\pi] = \infty$ \emph{and} $H\pi \cap K_N\pi \nleqslant [F_n, F_n]$, such that $\rrk(H\cap K_N)$ tends to infinity, as $N\to \infty$.

The construction of these $K_N$'s will be similar to that in Theorem~\ref{mainI}(ii), but with slight technical modifications in order to get the extra conditions. Take $K_N=\langle t^{{a'_1}}u_1^2, t^{{a'_2}}u_2^2, L_{K_N}\rangle\leqslant \Z^m\times F_n$, where the vectors $a'_1, a'_2\in \Z^m$ and the subgroup $L_{K_N}\leqslant \Z^m$ are to be determined. Note that $H\pi \cap K_N\pi = \langle u_1^2, u_2^2 \rangle \leqslant_\infty H\pi$, and also $H\pi \cap K_N\pi \nleqslant [F_n, F_n]$ as $u_1^2 \notin [F_n, F_n]$ (since $u_1 \notin [F_n, F_n]$ and $F_n/[F_n, F_n]= \Z^n$ is torsion-free).

The rest of the argument works in a parallel way, just realizing that now $P=\left( \begin{smallmatrix} 2 & 0 & 0 & \cdots & 0 \\ 0 & 2 & 0 & \cdots & 0 \end{smallmatrix}\right)\in M_{2\times n_1}(\Z)$ and so,
 $$
R=PA-P'A'=\left(\begin{array}{c} 2a_1 \\ 2a_2 \end{array} \right)-\left( \begin{array}{c}
a'_1 \\ a'_2 \end{array}\right) = \left( \begin{array}{c} 2a_1-a'_1 \\ 2a_2-a'_2 \end{array}
\right)\in M_{2\times m}(\Z),
 $$
Choosing the vectors $a'_1=2a_1-b_{s+1}$ and $a'_2=2a_2$, and the subgroup $L_{K_N}= \langle Nb_{s+1}\rangle \leqslant \Z^m$, the rest of the proof proceeds verbatim.

\textit{(ii)-$\leqslant$}. In order to show the inequality $\di'_G(H)\leqslant l\di'_{F_n}(H\pi)$, let us fix an arbitrary subgroup $K\leqslant_{fg} G$, assume $H\cap K\leqslant_{fg} G$ \emph{and also} $[H\pi : H\pi \cap K\pi] =\infty$ \emph{and} $H\pi \cap K\pi \nleqslant [F_n:F_n]$, and we have to prove that $\rrk(H\cap K)/\rrk(K)\leqslant l\di'_{F_n}(H\pi)$. Exactly the same arguments as in Theorem~\ref{mainI}~(iii) work here, with the caution that the inequality in Equation~\eqref{lost1} is still true with $\di_{F_n}(H\pi)$ replaced by the (eventually smaller) value $\di'_{F_n}(H\pi)$, since the involved subgroup $K\pi$ further satisfies $H\pi \cap K\pi \leqslant_\infty H\pi$ and $H\pi \cap K\pi \nleqslant [F_n, F_n]$, by construction.

\textit{(ii)-$\geqslant$}. By hypothesis, $\rk(H\pi )\geqslant 2$ and so, the Stallings' graph $\Gamma(H\pi )$ has at least one vertex $p$ of degree bigger than 2. Without loss of generality, we can assume that it is the basepoint $\odot$ who has degree at least $3$: in fact, let $w\in F_n$ be the label of any path from $\odot$ to $p$ and, replacing $H$ by $H^w$ (and so, $H\pi$ by $H^w\pi =(H\pi)^w$), the inequality to prove does not change; see Lemma~\ref{conjugation}.

Let $\{ t^{a_1}u_1,\ldots, t^{a_{n_1}}u_{n_1}, t^{b_1}, \ldots ,t^{b_m}\}$ be a basis for the subgroup $H\leqslant_{fg} G$, where $a_1,\ldots,a_{n_1}\in \Z^m$, $\{u_1,\ldots, u_{n_1}\}$ is a free basis for $H\pi$, $u_1\not\in [F_n,F_n]$, and $\{ b_1,\ldots ,b_m\}$ is a free-abelian basis for $L_H\leqslant_{l} \Z^m$; in particular, $\rk(H\pi)=n_1\geqslant 2$, $\rk(L_H)=m$, and $\rk(H)=n_1+m$.

In order to prove the inequality, $\di'_G(H) \geqslant l\di'_{F_n}(H\pi)$, we fix $\epsilon >0$ and will construct a subgroup $K_\epsilon \leqslant_{fg} G$ satisfying $H \cap K_\epsilon\leqslant_{fg} G$, $[H\pi : H\pi \cap K_\epsilon\pi] =\infty$, $H\pi \cap K_\epsilon\pi \nleqslant [F_n, F_n]$ and, furthermore, $\trk(H \cap K_\epsilon)/\trk(K_\epsilon) > l\di'_{F_n}(H\pi)- \epsilon$. For any candidate $K\leqslant G$, equations~\eqref{lost1}, \eqref{11}, and~\eqref{lost2} above (with $\di'_{F_n}$ instead of $\di_{F_n}$) contain all possible reasons for which the quotient $\trk(H \cap K)/\trk(K)$ may be less than $l\di'_{F_n}(H\pi)$, namely:
\begin{enumerate}
    \item[(I)] $\trk(H\pi \cap K\pi)/\trk(K\pi) \leqslant \di'_{F_n}(H\pi)$;
    \item[(II)] $\frac{l''\di'_{F_n}(H\pi)\rrk(K\pi)+\rk(L_H\cap L_K)}{\rrk(K\pi)+\rk(L_K)}\leqslant \frac{l''\di'_{F_n}(H\pi)\rrk(K\pi)}{\rrk(K\pi)}$;
    \item[(III)] $l'\leqslant l$;
    \item[(IV)] $l''\leqslant l'$.
\end{enumerate}
Choosing $K\pi$ so that $\trk(H\pi \cap K\pi)/\trk(K\pi) >\di'_{F_n}(H\pi)-\epsilon$ we can make the inequality in~(I) arbitrarily tight; choosing $L_K=0$ inequalities~(II) and~(III) become equalities; and, finally, if the linear map $R\colon \Z^{n_3}\to \Z^m$ from diagram~\eqref{m} is onto then inequality~(IV) becomes an equality. In view of these, we claim that

\begin{claim}\label{hard}
Given $\epsilon>0$, there exists $M\leqslant_{fg} F_n$ (with a free basis $\{v_1, \ldots ,v_{n_2}\}$), and there exist vectors $a'_1, \ldots, a'_{n_2}\in \Z^m$ such that:
\begin{itemize}
\item[(i)] $[H\pi : H\pi \cap M]=\infty$;
\item[(ii)] $H\pi \cap M \nleqslant [F_n, F_n]$;
\item[(iii)] $\trk(H\pi \cap M)/\trk(M) >\di'_{F_n}(H\pi)- \epsilon$;
\item[(iv)] $R=PA-P'A'\colon \Z^{n_3}\to \Z^m$ is onto, where $P,A,P'$ are the matrices appearing in diagram~\eqref{m} and $A'$ is the matrix with rows $a'_1, \ldots, a'_{n_2}\in \Z^m$.
\end{itemize}
\end{claim}

Observe that the existence of $M$ satisfying~(i), (ii), and~(iii) is immediate from the definition of $\di'_{F_n}(H\pi)$. Assuming, on top of these, condition~(iv) is more tricky: the choice of $M$ determines the ranks $n_2=\rk(M)$ and $n_3=\rk(H\pi \cap M)$, and it could very well happen that $n_3<m$, making then impossible to choose the vectors $a'_1, \ldots, a'_{n_2}\in \Z^m$ in such a way that $R$ is onto. This situation forces us to manipulate $M$ and make sure to get $n_3$ big enough, to have enough freedom, to choose $A'$, so that $R$ is onto; and all this without loosing the $\epsilon$ inequality (iii) (neither~(i) nor~(ii)). Here is where the extra technical conditions~(i) and~(ii) added to the definition of restricted degree of inertia are going to play a crucial role. Let us postpone the proof of the claim and continue with the main argument.

Given $\epsilon>0$, apply Claim~\ref{hard} to $\epsilon/l$: we get $M=\langle v_1, \ldots ,v_{n_2}\rangle\leqslant_{fg} F_n$, and vectors $a'_1, \ldots, a'_{n_2}\in \Z^m$ satisfying~(i), (ii), $\trk(H\pi \cap M)/\trk(M) >\di'_{F_n}(H\pi)- \epsilon/l$, and~(iv). The subgroup $K_{\epsilon}=\langle t^{a_1'}v_1,\ldots, t^{a_{n_2}'} v_{n_2} \rangle \leqslant_{fg} G$ satisfies $K_{\epsilon}\pi=M$ and $L_{K_{\epsilon}}=0$ hence,
\begin{itemize}
\item $H\cap K_\epsilon \leqslant_{fg} G$, since $(H\cap K_{\epsilon})\pi =(L_H+L_{K_{\epsilon}})R^{-1}\rho_3^{-1}\leqslant_{l''} H\pi \cap K_{\epsilon}\pi$,
\item $[H\pi : H\pi \cap K_\epsilon\pi ]=[H\pi : H\pi \cap M] =\infty$,
\item $H\pi \cap K_\epsilon\pi =H\pi \cap M\nleqslant [F_n, F_n]$,
\end{itemize}
and also
 $$
\frac{\trk(H\cap K_\epsilon)}{\trk(K_\epsilon)} =\frac{\trk((H\cap K_\epsilon)\pi)+\rk (L_H\cap L_{K_\epsilon})}{\trk(K_\epsilon\pi)+\rk (L_{K_\epsilon})} =\frac{\trk((H\cap K_\epsilon)\pi)}{\trk(K_\epsilon\pi)}=\frac{l''\trk(H\pi \cap K_\epsilon \pi)}{\trk(K_\epsilon\pi)}=
 $$
 $$
=\frac{l'\trk(H\pi \cap M)}{\trk(M)}=\frac{l\trk(H\pi \cap M)}{\trk(M)}>l(\di'_{F_n}(H\pi)- \epsilon/l)=l\di'_{F_n}(H\pi)- \epsilon.
 $$
Therefore, $\di'_G(H) \geqslant l\di'_{F_n}(H\pi)$ as we wanted to prove.
\end{proof}

Before proving Claim~\ref{hard}, we need to develop several lemmas about intersections of subgroups of $F_n$. A well-known tool for understanding these intersections is the pull-back of graphs.

\begin{dfn}
Let $N,M\leqslant_{fg} F_n$ and consider its Stallings' graphs $\Gamma(N),\Gamma(M)$, respectively. Its \emph{direct product}, $\Gamma(N)\times \Gamma(M)$, is defined as the new graph having as set of vertices $V\Gamma(N)\times V\Gamma(M)$, set of $x_i$-labelled edges $E_{x_i}\Gamma(N) \times E_{x_i}\Gamma(M)$ (here, $E_{x_i}\Gamma$ denotes the set of edges in $\Gamma$ labelled by the letter $x_i\in X$), with the natural incidence functions $\iota (e, f)=(\iota e,\iota f)$ and $\tau (e,f)=(\tau e,\tau f)$, and with basepoint being the pair of basepoints $(\bp, \bp)$.

Clearly, $\Gamma(N)\times \Gamma(M)$ is folded, but neither connected nor free of degree one vertices, in general. The \emph{pull-back of $\Gamma(N)$ and $\Gamma(M)$}, denoted $\Gamma(N)\wedge \Gamma(M)$ is the result of trimming (i.e., repeatedly deleting vertices of degree one different from the basepoint) the connected component of the basepoint ($\bp, \bp$).
\end{dfn}

It is well known (see, for example, \cite{KM} for details) that $\Gamma(N)\wedge\Gamma(M)\simeq \Gamma(N \cap M)$, the Stallings' graph for $N \cap M$. In particular, if both $N,M$ are finitely generated then so is $N\cap M$; this is the Howson property for $F_n$.

\begin{dfn} Let $\Gamma(N)$ be the Stallings' graph for $N\leqslant_{fg} F_n$. For every vertex $p \in V\Gamma(N)$ and every element $w \in F_n$, we define $pw$ to be the terminal vertex of the unique reduced path $\gamma$ in $\Gamma(N)$ starting at $p$ and with label $w$, in case it exists; otherwise, $pw$ is \textit{undefined}. Note that $w \in N$ if and only if $\bp w$ is defined and equals $\bp$. Note also that $N$ has finite index in $F_n$ if and only if $\Gamma(N)$ is complete and if and only if $\bp w$ is defined in $\Gamma(N)$, for every $w \in F_n$.
\end{dfn}

\begin{lemma}\label{undefined 1}
Let $N,M\leqslant_{fg} F_n$, with the basepoint from $\Gamma(N)$ having degree at least 3. Then, $N\cap M$ has infinite index in $N$ if and only if there exists $w \in N$ such that $\bp w$ is undefined in $\Gamma(M)$.
\end{lemma}

\begin{proof}
Suppose $N\cap M$ has finite index in $N$; so, there is $r\geqslant 1$ such that, for every $w\in N$, $w^r\in N\cap M$. This means that, for every cyclically reduced $w\in N$, $(\bp,\bp)w^r$, and so $(\bp, \bp)w$, is defined in $\Gamma(N\cap M)$; hence, projecting to $\Gamma(M)$, $\bp w$ is defined in $\Gamma(M)$. For those $w\in N$ not cyclically reduced, write $w=u^{-1}\cdot w'\cdot u$, without cancellations, with $w'$ being cyclically reduced, and with $u\neq 1$; in this case, we just have that $(\bp,\bp)w^r=(\bp,\bp)u^{-1}w'^ru$, and so $(\bp,\bp)u^{-1}w'$, is defined in $\Gamma(N\cap M)$. But the basepoint in $\Gamma(N)$ has degree at least 3 so we can take a cyclically reduced $1\neq v\in N$ such that the product $w\cdot v=u^{-1}\cdot w'\cdot u\cdot v\in N$ has no cancellation and is cyclically reduced again; then $(\bp,\bp)wv$, and so $(\bp,\bp)w$, is defined in $\Gamma(N\cap M)$ and, hence, $\bp w$ is defined in $\Gamma(M)$.

For the other implication suppose that, for every $w\in N$, $\bp w$ is defined in $\Gamma(M)$, say $\bp w\in \{p_0 =\bp, p_1,\ldots ,p_r\} \subseteq V\Gamma(M)$. Choose a maximal tree $T$ in $\Gamma(M)$ and define $w_i =\ell(T[\bp, p_i]) \in F_n \text{ for } i=0,\ldots ,r$ (note that $w_0 =1$). The hypothesis tells us that $N \subseteq M \sqcup Mw_1 \sqcup \cdots \sqcup Mw_r$. Intersecting with $N$, we get $N \subseteq (N \cap M)\sqcup(N \cap M)v_1 \sqcup \cdots \sqcup(N \cap M)v_s$ for some $v_i \in N$ and $s \leqslant r$ (where we have deleted the possibly empty intersections). Since the other inclusion is immediate, we deduce that $N \cap M$ has finite index in $N$.
\end{proof}

\begin{prop}\label{p-expansion}[$p$-Expansion] Let $N, M\leqslant_{fg} F_n$, and suppose that $\rk(N)\geqslant 2$, the basepoint $\bp$ of $\Gamma(N)$ has degree at least $3$, and $N\cap M\leqslant_{\infty} N$. Then, for every $1\leqslant p\leqslant \infty$, there exist $p$ freely independent elements $w_1,\ldots ,w_p \in N$ such that $M\leqslant_{ff} M'=M*\langle w_1,\ldots ,w_p\rangle$ and $N\cap M\leqslant_{ff} (N\cap M)*\langle w_1,\ldots ,w_p\rangle \leqslant_{ff} N\cap M'\leqslant_{\infty} N$.
\end{prop}

\begin{proof}
Let $e_a, e_b, e_c$ be three different edges going out from $\bp$ in $\Gamma(N)$, $\iota e_a = \iota e_b = \iota e_c = \bp$, with pairwise different labels $a, b, c \in X^{\pm 1}$, respectively. By Lemma \ref{undefined 1}, there is $u_0 \in N$ such that $\bp u_0$ is undefined in $\Gamma(M)$. Realize $u_0$ as a reduced closed path $\gamma_0$ at $\bp$ in $\Gamma(N)$ and, without lost of generality, we can assume it finishes with $e_a^{-1}$. For $\alpha =a,b,c$, take a non-trivial reduced path $\eta_\alpha$ in the graph $\Gamma(N) \setminus \{ e_\alpha \}$ and closed at $\tau e_\alpha$ (there always exists such a path because $\rk(N) \geqslant 2$, even if $e_\alpha$ is a bridge since $\Gamma(N)$ has no vertices of degree $1$ except possibly $\bp$); now consider $\gamma_\alpha = e_\alpha \eta_\alpha e_\alpha^{-1}$ , a reduced closed path at $\bp$ in $\Gamma(N)$, beginning with $e_\alpha$ and ending with $e_\alpha^{-1}$ (so, its
label $u_\alpha = \ell(\gamma_\alpha) \in N$ is a reduced word on $X^{\pm 1}$ beginning with $\alpha$ and ending with $\alpha^{-1}$). Note then that the paths $\gamma_0, \gamma_1 =\gamma_0\gamma_b, \gamma_2 =\gamma_0\gamma_b\gamma_a, \gamma_3 =\gamma_0\gamma_b \gamma_a \gamma_b, \ldots$, and also the paths $\gamma_i\gamma_c\gamma_i^{-1}$, $i\geqslant 1$, are reduced as written; furthermore, all of them are closed paths at $\bp$ in $\Gamma(N)$ so, $w_i =\ell(\gamma_i \gamma_c \gamma_i^{-1}) \in N$, for all $i\geqslant 1$.

Now, let us extend the graph $\Gamma(M)$ by adding the necessary vertices and edges so that we can read all the paths $\gamma_i\gamma_c \gamma_i^{-1}$ from $\bp$, $i=1,\ldots ,p$: since $\bp u_0$ was undefined in $\Gamma(M)$, possibly an initial segment of $\gamma_0$ is readable in $\Gamma(M)$ but not the entire path, forcing us to append at least a new edge sticking out from $\Gamma(M)$; behind it, we add the rest of the construction, see Fig.~\ref{expan} (this means adding infinitely many new vertices and edges in the case $p=\infty$). Since the added paths are all reduced, the resulting graph presents no foldings and so it is a (possibly infinite) Stallings' graph, having $\Gamma(M)$ as a subgraph. Hence, $M$ is a free factor of its fundamental group, $M\leqslant_{ff} M'=M*\langle w_1,\ldots ,w_p\rangle$.

And let us compare the pull-backs $\Gamma(N)\wedge \Gamma(M)=\Gamma (N\cap M)$ and $\Gamma(N)\wedge \Gamma(M')=\Gamma (N\cap M')$. Since $w_i\in N$ for all $i\geqslant 1$, it is clear that $\Gamma(N)\wedge \Gamma(M')$ contains, as a subgraph, $\Gamma(N) \wedge \Gamma(M)$ with the same additions as in Fig.~\ref{expan}, and, possibly, more edges which we do not control. Therefore, $N\cap M\leqslant_{ff} (N\cap M)*\langle w_1,\ldots ,w_p\rangle \leqslant_{ff} N\cap M'$. 

Finally, to see that $N\cap M'$ is still an infinite index subgroup of $N$, observe that $w=\ell(\gamma_0\gamma_c)\in N$ is such that $\bp w$ is not defined in $\Gamma(M')$ (see Fig.~\ref{expan}) and apply Lemma~\ref{undefined 1}.
\end{proof}

\begin{figure}
\centering
\begin{tikzpicture} [every edge/.style={draw,dotted}]
\node [ cloud, cloud puffs=9, draw, minimum width=4cm, minimum height=2.5cm]
at (0,0) (m) {\textbf{$\Gamma(M)$}$\odot$};
\node [circle,draw] (s) at (3.3,-0.2) {};
\node [circle,draw] (p) at (5.3,-1) {};
\node [ellipse,draw] (q) at (5.3,1) {$\phantom{aa}\eta_c\phantom{aa}$};
\node [circle,draw] (p_1) at (8.3,-1) {};
\node [ellipse,draw] (q_1) at (8.3,1) {$\phantom{aa}\eta_c\phantom{aa}$};
\node [circle,draw] (m_1) at (9.8,-1) {};
\path[->] (p) edge node [left] {$e_c$} (q);
\path[->] (p_1) edge node [left] {$e_c$} (q_1);
\draw [] (node cs:name=p) -- (node cs:name=q);
\draw [] (node cs:name=p_1) -- (node cs:name=q_1);
\path[->] (m) edge [bend left] node [above] {$\gamma_0$} (s);
\draw [dotted,thick,] (0.4,0) -- (1.6,.8);
\draw [dotted,thick,] (10.6,-1) -- (10.82,-1);
\draw [dotted,thick,] (10.6,-1) -- (10.82,-1);
\draw [dotted,thick,] (11.32,-1) -- (11.54,-1);
\draw [dotted,thick,] (11.32,-1) -- (11.54,-1);
\draw [dotted,thick,] (11.92,-1) -- (12.14,-1);
\draw [dotted,thick,] (11.92,-1) -- (12.14,-1);
\path [->] (s) edge[bend left] node [above] {$\gamma_b$} (p);
\path [->] (p) edge[bend right] node [above] {$\gamma_a$} (p_1);
\path [->] (p_1) edge[bend left] node [above] {$\gamma_b$} (m_1);
\end{tikzpicture}
\caption{Expansion of $\Gamma(M)$.}\label{expan}
\end{figure}
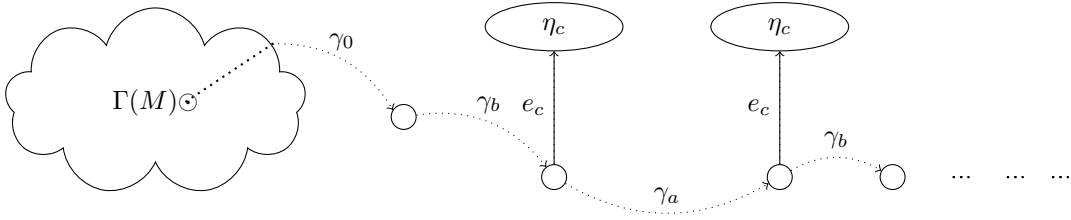

\begin{obs}\label{index}
Let $K\trianglelefteq F_n$ be a normal subgroup of $F_n$. For any $M\leqslant F_n$, $KM=\langle K,M\rangle$.
\end{obs}

\begin{proof}
It is obvious that $KM=\{ km \mid k\in K,\, m\in M\}\subseteq \langle K,M\rangle$. For the other inclusion note that, by normality, given $m\in M$ and $k\in K$, there exists $k'\in K$ such that $mk=k'm$. Repeated applications of this fact converts an arbitrary expression $k_1m_1\cdots k_rm_r\in\langle K,M\rangle$, with $k_i\in K$, $m_i\in M$, into a single product $km$, $k\in K$, $m\in M$. Therefore, $KM=\langle K,M\rangle$.
\end{proof}

\begin{lemma}\label{g}
Let $G$ be a group and $N,M\leqslant G$. Then, $[N:N\cap M]\leqslant [G:M]$, with equality if $MN=G$. If additionally $[N:N\cap M]$ is finite, the equality holds if and only if $MN=G$.
\end{lemma}

\begin{proof}
Let $G=\sqcup_{i\in I}M g_i$ be the coset decomposition of $G$ modulo $M$, where $|I|=[G:M]\leqslant \infty$. Intersecting with $N$ (and removing the possibly empty terms), we have $N=\sqcup_{i\in I} (N\cap Mg_i) =\sqcup_{i'\in I'} (N\cap M)n_i$, for some $I'\subseteq I$ and $n_i\in N$. So, $[N:N\cap M]=|I'|\leqslant |I|=[G:M]$.

Furthermore, for $g\in G$, $Mg$ intersects $N$ non-trivially if and only if $g\in MN$. So, if $G=MN$ then $I'=I$ and $[N:N \cap M]=|I'|=|I|=[G:M]$ (with the converse being also true whenever $|I'|<\infty$).
\end{proof}

\begin{cor}\label{index 1}
Let $K \trianglelefteq_{d} F_n$, and $M \leqslant F_n$, then $[M:M\cap K]=d$ if and only if $\langle K,M\rangle=F_n$. \qed
\end{cor}

Let us now fix a letter $x_j\in X$, and an integer $d\in \Z$, and consider the particular normal subgroup $K^j_d =\{ w\in F_n \mid |w|_{x_j} \in d\,\Z \}\unlhd F_n$, where $|w|_{x_j}$ denotes the $x_j$-th coordinate of the abelianization of $w$. The Stallings' graph $\Gamma(K^j_d)$ is depicted in Fig.~\ref{fig K_d^z} (the loops at each vertex representing all the $n-1$ remaining letters). It is clear that $K^j_d \trianglelefteq_{d} F_n$.
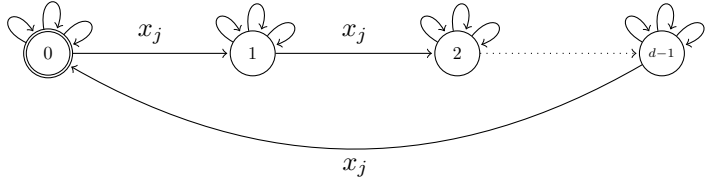
\begin{figure}
\centering
\begin{tikzpicture}[shorten >=1pt,node distance=4cm,auto]
\node [state,accepting,scale=0.70] (1) {$0$};
\node[state, scale=0.68] (2) [ right of = 1] {$1$};
\node[state, scale=0.68] (3) [ right of = 2] {$2$};
\node[state, scale=0.68] (4) [ right of = 3] {$_{d-1}$};
\path[->] (2) edge [in=70,out=100,loop] node[auto] {} ();
\path[->] (2) edge  [in=20,out=50,loop] node[auto] {} ();
\path[->] (2) edge  [in=120,out=150,loop] node[auto] {} ();
\path[->] (3) edge [in=70,out=100,loop] node[auto] {} ();
\path[->] (3) edge  [in=20,out=50,loop] node[auto] {} ();
\path[->] (3) edge  [in=120,out=150,loop] node[auto] {} ();
\path[->] (1) edge [in=70,out=100,loop] node[auto] {} ();
\path[->] (1) edge  [in=20,out=50,loop] node[auto] {} ();
\path[->] (1) edge  [in=124,out=154,loop] node[auto] {} ();
\path[->] (4) edge [in=70,out=100,loop] node[auto] {} ();
\path[->] (4) edge  [in=20,out=50,loop] node[auto] {} ();
\path[->] (4) edge  [in=120,out=150,loop] node[auto] {} ();
\path[->] (1) edge node {$x_j$} (2);
\path[->] (2) edge node {$x_j$} (3);
\path[->,dotted] (3) edge node {} (4);
\path[->] (4) edge [bend left] node {$x_j$} (1);
\end{tikzpicture}
\caption{Stallings' graph for $K^j_d$}
\label{fig K_d^z}
\end{figure}

\begin{lemma}\label{index 5}
Let $M\leqslant_{fg} F_n$, and fix a letter $x_j\in X$ and an integer $d\in \Z$. Then the following conditions are equivalent:
\begin{enumerate}
\item[(a)] $K^j_d M=\langle K^j_d, M\rangle=F_n$;
\item[(b)] $M\cap K^j_d \leqslant_d M$;
\item[(c)] there exists a word $m\in M$ such that $gcd(|m|_{x_j},d)=1$;
\item[(d)] the direct product $\Gamma(M)\times \Gamma(K^j_d)$ is connected.
\end{enumerate}
\end{lemma}

\begin{proof}
$(a)\Leftrightarrow (b)$. True by Observation~\ref{index} and Corollary~\ref{index 1}.

$(a)\Rightarrow (c)$. From the hypothesis, $x_j\in F_n$ can be written as $x_j=km$, for some $k\in K^j_d$ and some $m\in M$. Thus, $1=|x_j|_{x_j}=|k|_{x_j}+|m|_{x_j}=\lambda d+|m|_{x_j}$ for some $\lambda \in \Z$ and hence, $gcd(|m|_{x_j},d)=1$.

$(c)\Rightarrow (d)$. In the direct product $\Gamma(M)\times \Gamma(K^j_d)$, consider the $d$ full subgraphs $\Delta_i$, $i=0, \ldots, d-1$, whose vertices are $V\Delta_i =\{ (p,i) \mid p\in V\Gamma(M)\}$. In order to see that $\Gamma(M)\times \Gamma(K^j_d)$ is connected, we shall prove the existence of a path from $(\bp,i)$ to $(\bp,i+1)$, for every $i$ (indices modulo $d$), and that each $\Delta_i$ is connected.

In fact, by hypothesis there exists $m\in M$ and $\alpha, \beta \in \Z$ such that $gcd(|m|_{x_j},d)=\alpha |m|_{x_j}+\beta d=1$. Let $\gamma_m$ denote the path in $\Gamma(M)$, closed at $\bp$, whose label is $m$, and note that $\ell (\gamma_m^{\alpha})=m^{\alpha}$. Furthermore, in the complete graph $\Gamma(K^j_d)$, $i\cdot m^\alpha$ is defined and equals $i+1 \mod d$ (by Bezout's equality above). Hence, there is also a path labelled $m^\alpha$ in $\Gamma(M)\times \Gamma(K^j_d)$ from $(\bp, i)$ to $(\bp, i+1)$, for every $i=0,\ldots ,d-1$.

On the other hand, let $p$ be any arbitrary vertex in $\Gamma(M)$. As $\Gamma(M)$ is connected, there is a path, say $\gamma$, from $\iota(\gamma)=\bp$ to $\tau(\gamma)=p$. Let $w=\ell(\gamma) \in F_n$, let $s=|w|_{x_j}$, and consider the path $\gamma_{m^{-s\alpha}w}$ starting at $\bp$, reading $m^{-s\alpha}w$, and ending at $p$. Since $|m^{-s\alpha }w|_{x_j} =-s\alpha |m|_{x_j}+|w|_{x_j}=0 \mod d$, and $K^j_d$ is complete, $m^{-s\alpha }w$ is the label of a closed path in $\Gamma (K^j_d)$ at any vertex. Hence, there is a path in $\Delta_i\subseteq \Gamma(M)\times \Gamma(K^j_d)$ starting at $(\bp, i)$ and ending at $(p, i)$.

$(d)\Rightarrow (b)$. The graph $\Gamma (K^j_d)$ has exactly $d$ vertices, and $d$ edges labelled by each letter; see Fig.~\ref{fig K_d^z}. As $\Gamma(M)\times \Gamma(K^j_d)$ is connected, we have 
 \begin{align*}
\trk(M\cap K^j_d) & =-|V\Gamma(M\cap K^j_d)|+|E\Gamma(M\cap K^j_d)| \\ & =-|V\Big( \Gamma(M)\times \Gamma(K^j_d)\Big) |+|E\Big( \Gamma(M)\times \Gamma(K^j_d)\Big) | \\ & =-d|V\Gamma(M)|+d|E\Gamma(M)| \\ & =d\trk(M).
 \end{align*}
Hence, by Schreier index formula, $M\cap K^j_d\leqslant_{d} M$ (and not less).
\end{proof}

With these ingredients we can already proof Claim~\ref{hard}.

\begin{proof}[Proof of Claim~\ref{hard}]
We are given a subgroup $H\leqslant_{fg} G$ with $\rk(H\pi)\geqslant 2$, $H\pi \not\in [F_n, F_n]$, with the basepoint in $\Gamma(H\pi)$ having degree at least 3, with $L_H=H\cap \Z^m\leqslant_l \Z^m$ (and so, $\rk(L_H)=m$), and with basis $\{ t^{a_1}u_1,\ldots, t^{a_{n_1}}u_{n_1}, t^{b_1}, \ldots ,t^{b_m}\}$, such that $u_1\not\in [F_n,F_n]$; we are also given $\epsilon >0$, and we have to find a subgroup $M\leqslant_{fg} F_n$ and vectors $a'_1, \ldots, a'_{n_2}\in \Z^m$, $n_2=\rk(M)$, satisfying simultaneously conditions (i)-(iv). Let us distinguish two cases.

\noindent \textit{\bf \boldmath Case-1: $\di'_{F_n}(H\pi)>1$}. Making $\epsilon$ smaller if necessary, there always exists a subgroup $M_1\leqslant_{fg} F_n$ such that $[H\pi : H\pi\cap M_1]=\infty$, $H\pi\cap M_1\nleqslant [F_n, F_n]$ and
 \begin{equation}\label{index 2}
\frac{\trk(H\pi \cap M_1)}{\trk(M_1)}>\di_{F_n}'(H\pi)-\frac{\epsilon}{2}>1.
 \end{equation}
Hence, both reduced ranks are $\rrk(H\pi \cap M_1)\geqslant 1$ and $\rrk(M_1)\geqslant 1$ (recall that, in the definition of $\di'_{F_n}(\cdot)$, $0/0$ is understood to be $1$). As $H\pi \cap M_1 \nleqslant [F_n, F_n]$, there exists $v\in H\pi\cap M_1$ and a letter $x_j\in X$ such that $|v|_{x_j}=\lambda \neq 0$. Write $\lambda=p_1^{\alpha_1}\cdots p_r^{\alpha_r}$, where each $p_i$ is a prime divisor of $\lambda$. Now choose a big enough prime number $d\gg 0$, such that $\gcd(\lambda, d)=1$ and $d>2m\di'_{F_n}(H\pi)/\epsilon$. We have $\epsilon \rrk(M_1 )\big( d\rrk(M_1 )+m\big) \geqslant \epsilon d\rrk(M_1 )>2m\di'_{F_n}(H\pi)\rrk(M_1 )\geqslant 2m\rrk(H\pi \cap M_1 )$ and so,
 \begin{align*}
2\big( d\rrk(H\pi \cap M_1 )+m\big) \rrk(M_1 ) & \geqslant 2d\rrk(H\pi \cap M_1)\rrk(M_1 ) \\ & > 2d\rrk(H\pi \cap M_1 )\rrk(M_1 )+2m\rrk(H\pi \cap M_1 )-\epsilon \rrk(M_1 )\big( d\rrk(M_1 )+m\big) \\ & = 2\rrk(H\pi \cap M_1 )\big( d\rrk(M_1 )+m\big)-\epsilon \rrk(M_1)(d\rrk(M_1 )+m) \\ & = \big( 2\rrk(H\pi \cap M_1)-\epsilon \rrk(M_1)\big) \big( d\rrk(M_1 )+m\big).
 \end{align*}
Dividing both sides by $2\big( d\rrk(M_1 )+m\big)\rrk(M_1 )\neq 0$ we get
\begin{equation}\label{star equ}
\frac{d\trk(H\pi \cap M_1 )+m}{d\trk(M_1 )+m}>\frac{\trk(H\pi \cap M_1 )}{\trk(M_1 )}-\frac{\epsilon}{2}.
\end{equation}

Now, consider $K^j_d\unlhd_d F_n$ and $M_2:= M_1\cap K^j_d$. Since the element $v\in H\pi\cap M_1\leqslant M_1$ satisfies $\gcd(|v|_{x_j}, d)=\gcd (\lambda, d)=1$, Lemma~\ref{index 5} tells us that $M_2\leqslant_d M_1$ and $H\pi \cap M_2=(H\pi \cap M_1) \cap K^j_d \leqslant_d H\pi \cap M_1$, and also $\Gamma(M_1)\wedge \Gamma(K^j_d)=\Gamma(M_1)\times \Gamma(K^j_d)$ and $\Gamma(H\pi\cap M_1)\wedge \Gamma(K^j_d)=\Gamma(H\pi\cap M_1)\times \Gamma(K^j_d)$; note further that $\Gamma(H\pi\cap M_1)\times \Gamma(M_2)$ is not necessarily connected, but its connected component containing the basepoint (after trimming, if necessary) coincides with $\Gamma(H\pi\cap M_1)\times \Gamma(K^j_d)$, both being the Stallings' graph of $(H\pi\cap M_1)\cap K^j_d=(H\pi\cap M_1)\cap M_2$; see 
Fig.~\ref{pbs}, where every graph in the second and third columns is the pull-back of the one left to it and the one above it. Moreover, from the hypothesis we have $H\pi\cap M_1 \leqslant_\infty H\pi$ and so, $H\pi \cap M_2 \leqslant_\infty H\pi$.
\begin{figure}
\centering
 $$
\begin{array}{ccccccc}
& & \Gamma(F_n) & _{_d \geqslant} & \Gamma(K^j_d) & & \\ & & \Gamma(M_1) & _{_d \geqslant} & \Gamma(M_2) & _{\leqslant_{ff}} & \Gamma(M) \\ \Gamma(H\pi) & _{_{\infty} \geqslant} & \Gamma(H\pi\cap M_1) & _{_d \geqslant} & \Gamma(H\pi\cap M_2) & _{\leqslant_{ff}} & \Gamma(H\pi\cap M)
\end{array}
 $$
\caption{Diagram of pull-backs}
\label{pbs}
\end{figure}

Now, let us apply Proposition~\ref{p-expansion}, with $p=m$, to the pull-back of $\Gamma(H\pi)$ and $\Gamma(M_2)$ (which equals $\Gamma(H\pi\cap M_2)$): we are under the hypothesis ($\rk(H\pi)\geqslant 2$, the basepoint $\bp$ of $\Gamma(H\pi)$ has degree at least $3$, and $H\pi\cap M_2\leqslant_{\infty} H\pi$) so we can perform an $m$-expansion to $M_2$ and get $m$ freely independent elements $w_1, \ldots , w_m \in F_n$ such that $M_2\leqslant_{ff} M:=M_2*\langle w_1, \ldots , w_m \rangle$ and $H\pi \cap M_2 \leqslant_{ff} (H\pi \cap M_2)*\langle w_1, \ldots , w_m \rangle \leqslant_{ff} H\pi \cap M\leqslant_{\infty} H\pi$; in particular, $n_2=\trk(M)=\trk(M_2)+m$ and $\trk(H\pi \cap M)\geqslant \trk(H\pi \cap M_2)+m$. This is our candidate subgroup $M\leqslant_{fg} F_n$: in fact, (i) $H\pi\cap M \leqslant_{\infty} H\pi$; (ii) $H\pi \cap M_1 \not\leqslant [F_n,F_n]$ so, $H\pi \cap M_2 \not\leqslant [F_n,F_n]$ (since $F_n/[F_n,F_n]\simeq \Z^n$ is tor\-sion free) and $H\pi \cap M \not\leqslant [F_n,F_n]$; and (iii), by the Schreier index formula and equations~\eqref{index 2}-\eqref{star equ},
 $$
\frac{\trk(H\pi \cap M)}{\trk(M)}\geqslant \frac{\trk(H\pi \cap M_2)+m}{\trk(M_2)+m}= \frac{d\trk(H\pi \cap M_1)+m}{d\trk(M_1)+m} >\frac{\trk(H\pi \cap M_1 )}{\trk(M_1 )}-\frac{\epsilon}{2}> \di'_{F_n}(H\pi)-\epsilon.
 $$
It remains to choose appropriate vectors $a'_1, \ldots, a'_{n_2}\in \Z^m$ satisfying (iv). Take a free basis for $M_2$, $\{v_1, \ldots , v_k\}$, and extend it to a free basis for $M$, $\{v_1, \ldots , v_k, w_1, \ldots , w_m\}$. We have that $n_1=\rk(H\pi)$, $k=\rk(M_2)$, and $n_2=k+m=\rk(M)$. Similarly, as a free basis for $H\pi \cap M$, take a free basis for $H\pi \cap M_2$ (say, made of $q=\rk(H\pi\cap M_2)$ freely independent elements) followed by possibly some more, say $p\geqslant 0$, and finally followed by $\{w_1, \ldots , w_m\}$; we have $n_3 :=\rk(H\pi \cap M)=q+p+m\geqslant m$. Finally, consider the intersection diagram for $H\pi$ and $M$,
 \begin{equation}\label{int-big}
\begin{aligned}
\xy
(0,0)*+{H\pi \cap M}; (-25,0)*+{H\pi}; (25,0)*+{M};
{\ar@{_(->}_-{\iota} (0,0)*++++++++++{}; (-25,0)*++++{}};
{\ar@{^(->}^-{\iota'} (0,0)*++++++++++{}; (25,0)*++++{}};
(0,-20)*+{\ZZ^{q+p+m}}; (-25,-20)*+{\ZZ^{n_1}}; (25,-20)*+{\ZZ^{k+m}};
{\ar@{->>}^-{\rho_3} (0,0)*+++{}; (0,-20)*+++{}};
{\ar@{->>}_-{\rho_1} (-25,0)*+++{}; (-25,-20)*+++{}};
{\ar@{->>}^-{\rho_2} (25,0)*+++{}; (25,-20)*+++{}};
(-12.5,-10)*+{///};
(12.5,-10)*+{///};
{\ar_-{{P}} (-3,-20)*+++++{}; (-25,-20)*++++{}};
{\ar^-{{P}'} (3,-20)*+++++{}; (25,-20)*++++{}};
(0,-40)*+{\ZZ^{m}};
{\ar^-{{A}} (-25,-20)*++++{}; (0,-40)*++++{}};
{\ar_-{{A'}} (25,-20)*++++{}; (0,-40)*++++{}};
{\ar@[]^-{{R}} (0,-20)*++++{}; (0,-40)*++++{}};
\endxy
\end{aligned}
 \end{equation}
where $P\in M_{(q+p+m)\times n_1}(\Z)$ and $P'\in M_{(q+p+m)\times (k+m)}(\Z)$ are the abelianization of the inclusions $H\pi\cap M\hookrightarrow H\pi$ and $H\pi\cap M\hookrightarrow M$, respectively, where $A\in M_{n_1\times m}(\Z)$ is the matrix with rows $a_1, \ldots ,a_{n_1}$, and where $A'\in M_{(k+m)\times m}(\Z)$ is the matrix with rows $a'_1, \ldots ,a'_{k+m}$ to be determined in such a way that $R=PA-P'A'\colon \Z^{q+p+m}\to \Z^m$ becomes onto.

Note that, by construction, the first $q$ elements in the free basis for $H\pi\cap M$ are freely independent from $\{w_1,\ldots ,w_m\}$, and that $\{w_1, \ldots ,w_m\}$ are present in the last positions of the chosen bases for both $H\pi\cap M$ and $M$; therefore, $P'$ has the form
 $$
P'=\left( \begin{array}{c|c} * & 0 \\ \hline * & * \\ \hline 0 & I_m \end{array} \right).
 $$
Let $Q$ be the lower $m\times m$ block in $PA\in M_{(q+p+m)\times m}(\Z)$, and define
 $$
A'=\left( \begin{array}{c} 0 \\ \hline -I_m+Q \end{array} \right) \in M_{(k+m)\times m}(\Z).
 $$
Separating the rows in the natural blocks, we have that
  $$
R=PA-P'A'=\left( \begin{array}{c} * \\ \hline * \\ \hline Q \end{array} \right) -\left( \begin{array}{c|c} * & 0 \\ \hline * & * \\ \hline 0 & I_m \end{array} \right) \left( \begin{array}{c} 0 \\ \hline -I_m+Q \end{array} \right) =\left( \begin{array}{c} * \\ \hline * \\ \hline Q \end{array} \right) -\left( \begin{array}{c} 0 \\ \hline * \\ \hline -I_m+Q \end{array} \right) =\left( \begin{array}{c} * \\ \hline * \\ \hline I_m \end{array} \right)
 $$
is a surjective map from $\Z^{q+p+m}$ onto $\Z^m$, proving condition (iv). This concludes the proof for this case-1.

\noindent \textit{\bf \boldmath Case-2: $\di'_{F_n}(H\pi)=1$}. Recall that we already fixed a basis $\{t^{a_1}u_1,\ldots, t^{a_{n_1}}u_{n_1}, t^{b_1}, \ldots ,t^{b_m} \}$ for $H$, with $u_1 \not\in [F_n, F_n]$. Let $M=\langle u_1, {u_2}^{-1}u_1u_2, \ldots ,{u_2}^{-(m-1)}u_1{u_2}^{m-1} \rangle\leqslant_{fg} H\pi$, a subgroup with $n_2=\rk(M)=m$ already satisfying the first three required conditions: (i) $H\pi\cap M=M\leqslant_{\infty} H\pi =\langle u_1, \ldots ,u_{n_1}\rangle$; (ii) $u_1 \in H\pi\cap M=M\not\leqslant [F_n,F_n]$; and (iii) $\rrk(H\pi\cap M)/\rrk(M)=1>\di'_{F_n}(H\pi)-\epsilon$, independently from the given $\epsilon$.

Finally, we have to choose appropriate vectors $a_1', \ldots, a_m' \in \Z^m$ so that (iv) holds. To do this, look at the intersection diagram for $H\pi$ and $M$ (see Fig.~\ref{m}): we have $n_3=n_2 =m$ and
 $$
P=\left( \begin{array}{cccc} 1 & 0 & \hdots & 0 \\ 1 & 0 & \hdots & 0 \\ \vdots & \vdots & & \vdots \\ 1 & 0 & \hdots & 0 \end{array}\right)\in M_{m\times n_1}(\Z), \quad \quad P'=I_m \in M_{m\times m}(\Z)
 $$
and, therefore,
 $$
R=PA-P'A'=\left(\begin{array}{c} a_1 \\ \vdots \\ a_1 \end{array} \right)-A'\in M_{m\times m}(\Z)
 $$
will become the identity $I_m$, and so onto, after choosing $A'$ appropriately. This shows (iv) and concludes the proof.
\end{proof}

\noindent\textbf{Acknowledgements.} The first author thanks the support and hospitality from the Barcelona Graduate School of Mathematics and the Departament de Matem\`atiques of the Universitat Polit\`ecnica de Catalunya. Both authors are partially supported by the Spanish Agencia Estatal de Investigaci\'on, through grant MTM2017-82740-P (AEI/ FEDER, UE), and  also by the ``Mar\'{\i}a de Maeztu'' Programme for Units of Excellence in R\&D (MDM-2014-0445).

\end{document}